\documentclass[reqno]{amsart}
\usepackage{graphicx}
\usepackage{trajan}
\usepackage{multirow}
\usepackage{latexsym,array,delarray,amsthm,amssymb,epsfig}
\usepackage{amsmath}
\usepackage{yhmath}
\usepackage{mathdots}
\usepackage{MnSymbol}
\usepackage{tikz}
\usepackage{adjustbox}
\usepackage{rotating}
\usepackage{longtable}
\usepackage{caption}
\usepackage{wasysym}
\usetikzlibrary{shapes.misc,shadows}
\usepackage{tabularx}
\usepackage[shortlabels]{enumitem}
\setlist[enumerate]{noitemsep}

\theoremstyle{plain}
\newtheorem{thm}{Theorem}[section]

\newtheorem*{thm*}{Theorem}
\newtheorem*{lemma*}{Lemma}
\newtheorem*{prop*}{Proposition}
\newtheorem*{cor*}{Corollary}
\newtheorem*{conj*}{Conjecture}

\newtheorem{proposition}[thm]{Proposition}

\theoremstyle{definition}

\newtheorem{example}[thm]{Example}

\theoremstyle{remark}

\newtheoremstyle{exampstyle}
{2pt} 
{-3pt} 
{} 
{} 
{\bfseries} 
{.} 
{.5em} 
{} 
\theoremstyle{exampstyle}

\numberwithin{equation}{section}

\theoremstyle{definition}

\theoremstyle{remark}

\newcommand{\Z}{\mathbb Z} 
\newcommand{\Q}{\mathbb Q} 
\newcommand{\R}{\mathbb R} 

\newcommand{\bq}{\mathbf q}

\usepackage{placeins}

\makeatletter
\g@addto@macro\@floatboxreset\centering
\makeatother
\newcommand{\eref}[1]{Eq.~\ref{#1}}
\newcommand{\fref}[1]{Fig.~\ref{#1}}

\usepackage{pdfpages}
\usepackage{fancyhdr}
\fancypagestyle{mylandscape}{%
  \fancyhf{}
  \fancyfoot{
    \makebox[\textwidth][r]{
      \rlap{\hspace{\footskip}
        \smash{
          \raisebox{\dimexpr.5\baselineskip+.5\footskip+.5\textheight}{
            \rotatebox{90}{\thepage}}}}}}
}

\usepackage{pdflscape}
\usepackage{tikz}
\fancypagestyle{lscapedplain}{%
  \fancyhf{}
  \fancyfoot{%
    \tikz[remember picture,overlay]
      \node[outer sep=1cm,above,rotate=90] at (current page.east) {\thepage};}

}

\begin{document}
\setlength{\intextsep}{1pt} 
\setlength{\textfloatsep}{0pt} 
\setlength{\abovecaptionskip}{6pt}
\setlength{\belowcaptionskip}{6pt}

\date{}

\title{On Smith normal forms of $q$-Varchenko matrices}
\author{Naomi Boulware, Naihuan Jing, Kailash C. Misra}
\address{Department of Mathematics, Earlham College, Richmond, IN 47374, USA}
\email{boulwna@earlham.edu} 
\address{Department of Mathematics, North Carolina State University, Raleigh, NC 27695, USA}
\email{
jing@ncsu.edu, misra@ncsu.edu}
\thanks{}
\begin{abstract}
In this paper, we investigate $q$-Varchenko matrices for some hyperplane arrangements with symmetry in two and three dimensions, and prove that they have a Smith normal form over $\Z[q]$.  In particular, we examine the hyperplane arrangement for the regular $n$-gon in the plane and the dihedral model in the space and Platonic polyhedra.  In each case, we prove that the $q$-Varchenko matrix associated with the hyperplane arrangement has a Smith normal form over $\Z[q]$ and realize their congruent transformation matrices over $\Z[q]$ as well.
\end{abstract}
\maketitle
\section{Introduction}
\label{chap-zero}
Hyperplane arrangements are present in various problems in geometry, combinatorics, and algebra \cite{Gr, OT, PS, St1}. A hyperplane arrangement devides an affine space to disjoint regions. The set of hyperplanes separating two regions $R_i$ and $R_j$ is denoted by $sep(R_i, R_j)$.
In 1993, Varchenko defined a matrix associated with a hyperplane arrangement in \cite{V1} which helps to reveal the intrinsic combinatorial and algebraic properties of the hyperplane arrangement.
The Varchenko matrix has rows and columns indexed by the regions of the hyperplane arrangement, where each hyperplane is assigned an indeterminate $a_{H}$.  Then the $(i, j)$-entry of the Varchenko matrix is obtained by taking the product over all $a_{H}$ such that $H \in sep(R_{i}, \ R_{j})$.  Gao and Zhang \cite{GZ} have determined the necessary and sufficient conditions for the Varchenko matrix associated with a hyperplane arrangement to have a diagonal form. In \cite{St1}, Stanley noted that Gao and Zhang's result also holds for pseudosphere arrangements, which are a generalization of hyperplane arrangements.

$q$-Varchenko matrices come from these Varchenko matrices by replacing the $(i, j)$ entry by $q^{\#sep(R_{i}, \ R_{j})}$.  It is known \cite{DH}
 that the number of diagonal entries of the Smith normal form of the $q$- Varchenko matrix $V_q$ exactly divisible by
 $(q-1)^i$ is equal to the $i$-th Betti number of the complex complement of the arrangement. Applications of the invariant factors in the
 Smith normal form to combinatorial and discrete mathematics are discussed in  \cite{Sh}. Hyperplane arrangements also provide
interesting quantum integrable models \cite{V2}.
While the Smith normal form of a matrix is guaranteed to exist over a principal ideal domain such as $\Q[q]$,
the algorithm is not always computationally practical \cite{St2}. Combinatorially, it is important to know whether the congruent transformations can be carried out in $\Z[q]$.  In \cite{DH}, Denham and Hanlon illustrated two problems with implementation of the algorithm: one is that the size of the $q$-Varchenko matrix becomes extremely large very quickly in proportion to the number of hyperplanes in the arrangement; the second is that the degrees of the polynomial matrix entries ``blow up during the intermediate stages in the computation''.

Little is known about the existence of a Smith normal form of a matrix over rings that are not principal ideal domains.
One of the open problems is whether the hyperplane arrangement based on the root systems of type A has an integral Smith normal form \cite{St1}.

In this paper we consider an easier problem related to root systems. According to McKay \cite{Mc}, the simply-laced types of root systems are in one-to-one correspondence to the symmetry groups of
Platonic polyhedra \cite{Sl} and their degenerates such as the regular polygons and regular polyhedras. Explicitly, we will study
the related hyperplane arrangement models based on Platonic polyhedra and their degenerations. The cyclic model has been studied in a
recent paper \cite{CCM} as an example of
peelable hyperplane arrangement. In this paper we will use a different method to approach all hyperplane arrangement models based on regular
polyhedra, and show that all of the q-Varchenko matrices have the Smith normal forms over $\mathbb Z[q]$ and also the congruent transformations can be realized in $\mathbb Z[q]$ as well.

Using the symmetry of these arrangements, we give the $q$-Varchenko matrices, their Smith normal forms over $\Z[q]$, and the corresponding transition matrices.  Then, the methods used in the study of the $q$-Varchenko matrix for the octahedron arrangement are modified and applied to show that the $q$-Varchenko matrix for a hyperplane arrangement corresponding to a pyramid with a square base also has a Smith normal form over $\Z[q]$.  Lastly, we consider the hyperplane arrangement corresponding to a pyramid with a regular pentagonal base and show that its $q$-Varchenko matrix has a Smith normal form over $\Z[q]$.  In each case, we give the transition matrices and determine the Smith normal form.

The paper is organized as follows. In Section \ref{prelim} we discuss some basic notions and the hyperplane arrangements.
Next, in Section \ref{chap-two} we revisit the hyperplane arrangement of a regular $n$-gon in $\R^2$ which we call the cyclic model; we formulate an algorithm that allows us to use the symmetry of the arrangement to obtain its $q$-Varchenko matrix and determine the Smith normal form over $\Z[q]$ for an arbitrary $n$.  Then we move into $\R^3$ and define the dihedral model hyperplane arrangement; we obtain its $q$-Varchenko matrix using the symmetry of the arrangement, determine the Smith normal form over $\Z[q]$ for an arbitrary $n$, and give the transition matrices. Finally in Section \ref{poly models}, we consider the hyperplane arrangements in $\R^3$ corresponding to a tetrahedron, a cube, and an octahedron.


%

\section{Hyperplane arrangements}\label{prelim}

We first recall some basic notions about hyperplane arrangements following \cite{St2}. A real hyperplane is an $n-1$ dimensional subspace defined by a linear equation in $\mathbb R^n$.
Let $\mathcal{A} = \{h_{1}, h_{2}, \ldots, h_{k}\}$ be a set of real hyperplanes. If $\mathbb R^n\backslash \mathcal A$ is a disjoint union of open sets, then we call $\mathcal A$ a hyperplane arrangement, and the disjoint open subsets are referred as the regions of $\mathcal A$, and their collection is called the \textit{the set of regions} of $\mathcal{A}$, denoted by $\mathcal{R}(\mathcal{A})$.
The set of hyperplanes separating two regions $R, R'$ is denoted by $sep(R, \ R')$, and its cardinality $d(R,\ R') = \#sep(R, \ R')$ forms a metric on $\mathcal{R}(\mathcal{A})$.
If we fix a base region $R_{0}$, then the \textit{distance enumerator} of $\mathcal{A}$ with respect to $R_0$ is
\begin{equation*}
\displaystyle	D_{\mathcal{A}, R_0}(t) = \sum_{R \in \mathcal{R}(\mathcal{A})} t^{d(R_0, R)}.
\end{equation*}

%
Fix a base region $R_0$.  Then the regions of $\mathcal A$ form a partial order (weak order) defined by $R\preccurlyeq R'$ if $sep(R_0, R)\subseteq sep(R_0, R')$. This weak order $\mathcal R_{\mathcal A}$ is graded by distance.
The \textit{intersection poset} $L(\mathcal{A})$ is the set of all nonempty intersections of hyperplanes in $\mathcal{A}$, including the underlying space itself as the intersection over the empty set.  We have $x \le y$ in the intersection poset if $y \subseteq x$.  Therefore the underlying space is the minimal element, $\hat{0}$.

Let $x\in L(\mathcal A)$. The subarrangement $\mathcal{A}_x$ is defined as $\{h \in \mathcal{A} \ : \ x \subseteq h\}$.
The arrangement $\mathcal{A}^{x}$ is defined by $\mathcal{A}^{x} = \{x \cap h \neq \emptyset \ : \ h \in \mathcal{A} - \mathcal{A}_{x}\}.$
If $\mathcal{A} = \{h_{1}, h_{2}, \ldots, h_{n}\}$ is a hyperplane arrangement, then take $x = h_{k}$ and $\mathcal{A}^{h_{k}}$ is an arrangement in the affine subspace $h_{k}$:
\begin{equation} \label{A^{h_{k}}}
\displaystyle \mathcal{A}^{h_{k}} = \{h_{k} \cap h_{i} \neq \emptyset : h_{i} \in \mathcal{A} -\{h_{k}\}\}.
\end{equation}


Let $h_{k}$ be a hyperplane in $\mathcal{A}=\{h_{1}, h_{2}, \ldots, h_{m}\}$.  $\mathcal{A}'=\mathcal{A}-\{h_{k}\}$ is called the \textit{deleted arrangement}.  The arrangement in $h_{k}$ defined by $\mathcal{A}''=\{h_{i} \cap h_{k} : h_{i} \in \mathcal{A}'\}$ is called the \textit{restricted arrangement}.  Then the triple $(\mathcal{A}, \mathcal{A}', \mathcal{A}'')$ can be used to recursively solve the problem of counting the number of regions of $\mathcal{A}$:  the \textit{number of regions} of $\mathcal{A}$ equals the number of regions of $\mathcal{A}'$ plus the number of regions of $\mathcal{A}''$ \cite{OT}. The following result is easy to see.
\begin{proposition} \label{sep set}
	$\displaystyle sep\left(R, R'\right) = \left(sep\left(R_0, R\right) \cup sep\left(R_0, R'\right)\right) - \left(sep\left(R_0, R'\right) \cap sep\left(R_0, R\right)\right).$
\end{proposition}

The hyperplane arrangement of a regular $n$-gon is the hyperplane arrangement given by taking the edges of a regular $n$-gon and extending them to lines in $\R^2$.  Similarly, the hyperplane arrangement of a regular polytope is the hyperplane arrangement given by taking the sides of a regular polytope and extending them to planes in $\R^3$.
%


\section{Symmetry Models}\label{chap-two}
In this section we will develop a simple method to compute the $q$-Varchenko matrix of a hyperplane arrangement and show that the $q$-Varchenko matrix has an integral Smith normal form for the
simple example of the cyclic model. In particular, we will consider the hyperplane arrangement of the affine span of the facets of special polytopes.  

\subsection{The Cyclic Model} The Smith normal form for the $q$-Varchenko matrix associated with the hyperplane arrangement of a regular $n$-gon, $\mathcal{A} = \{h_{1}, h_{2}, \ldots, h_{n}\}$, was presented in \cite{CCM}.  Here the hyperplanes consists of all facets of the regular $n$-polygons. Our approach to finding the $q$-Varchenko matrix $V_{q}$ is different and meant to introduce a method to treat other models.  We refer to the method of arranging hyperplane arrangements of regular $n$-gons as the cyclic model $C_n$,
using the same symbol for the cyclic group of order $n$. 

Take the base region $R_0$ to be the center $n$-gon region, so $\#sep\left(R_0, R_0\right) = 0$.
The distance enumerator for $\mathcal{A}$ (with respect to $R_{0}$) is
\begin{equation} \label{p} D_{\mathcal{A}, R_0}(t) = 1+\sum_{k=1}^{p} n t^k, \ \mbox{ where } \ p = \lfloor \frac{n+1}{2} \rfloor.
\end{equation}

For positive integer $r$, the $n$-dimensional vector $\bq^r=(q^r, q^r, \ldots, q^r)$, is sometimes denoted by
$\bq_n^r$ to indicate its size if needed. We also use the following notations for the $i$-dimensional vectors:
\begin{align}
(q^r)_i&=(q^r, q^{r-2}, q^{r-4}, \ldots, q^{r-2i+2}),\\
(q^r)^i&=(q^r, q^{r+2}, q^{r+4}, \ldots, q^{r+2i-2}).
\end{align}
For $i<j$, we also define the $(n-j+i+1)$-dimensional vector:
\begin{equation}
[q^{i-j}]_{n-j+i+1}=(q^{i-j}, q^{i-j+2}, \ldots, q^{i+j-2}, \underbrace{q^{i+j}, \ldots, q^{i+j}}_{n+1-i-j}, q^{i+j-2}, \ldots, q^{i-j})
\end{equation}

We will denote the transpose of a matrix $M$ by $M^t$.

Recall that the circulant matrix $C = C(a_1, a_2, \ldots, a_n)$ is defined by:
\begin{equation}
C(a_1, a_2, \ldots, a_n)=\begin{pmatrix} a_1 & a_2 & a_3 &\cdots & a_n\\ a_n & a_1 & a_2&\cdots & a_{n-1}\\a_{n-1} & a_n & a_1&\cdots & a_{n-2}\\
\cdot & \cdot & \cdot & \cdots & \cdot \\a_2 & a_3 & a_4 &\cdots & a_{1}\end{pmatrix}.
\end{equation}

The reverse circulant matrix $RC = RC(a_1, a_2, \ldots, a_n)$ is defined by:
\begin{equation}
RC(a_1, a_2, \ldots, a_n)=\begin{pmatrix} a_1 & a_2 & a_3 &\cdots & a_n\\ a_2 & a_3 & a_4&\cdots & a_{1}\\a_{3} & a_4 & a_5&\cdots & a_{2}\\
\cdot & \cdot & \cdot & \cdots & \cdot \\a_n & a_{1} & a_{2} &\cdots & a_{n-1}\end{pmatrix}.
\end{equation}
So for any integer $1 \le a \le n$, the $(x', x)$-entry of $RC_{ij}$ is equal to the $(\overline{x'+a}, \overline{x-a})$-entry.  The overline refers to modulo $n$.

\begin{thm} \label{Cn Vq} For a natural integer $n$, let $p$ be as in \eqref{p}.
The $(1+np) \times (1+np)$ $q$-Varchenko matrix for the $C_{n}$ arrangement has the following block form:
\setlength{\arraycolsep}{6pt}
\setlength{\extrarowheight}{-1pt}
\begin{equation}\label{e:qVar}
V_q=\left(\begin{array}{cc ccc}1 & Q_1 & Q_2 & \dots & Q_p\\
	Q^{t}_{1} & C_{11} & C_{12} &  \dots & C_{1p} \\
	Q^{t}_{2} & C_{21} & C_{22}  & \dots  & C_{2p}\\
	\vdots & \vdots & \vdots & \ddots & \vdots\\
	Q^{t}_{p} & C_{p1} & C_{p2} & \dots & C_{pp}
	\end{array}\right)
\end{equation}
where $Q_{r}$ is the row vector $\bq^r_n$ 
and $C_{ij}$ ($1 \le i \le j \le p$) is the circulant matrix
given by 
\begin{equation}  \label{Cn Eij form}
	C_{i j} = C ((q^{j-i})^i, \bq^{i+j}_{n+1-i-j}, (q^{j-i+2(i-1)})_{i-1}, \bq^{j-i}_{j-i}).
 \end{equation}
\end{thm}
\begin{proof} By definition the Varchenko matrix encodes the intersection numbers among separable regions in the model: $V_q=(q^{\#sep(R_{i}, \ R_{j})})$, where $\#sep(R_{i}, \ R_{j})$ is the number of hyperplanes separating region $R_i$ from region $R_j$.

Now let's carefully count the separating numbers $\#sep(R_{i}, \ R_{j})$.
For $1 \le k \le p$, there are $n$ regions $R$ such that $\#sep\left(R_0, R\right) = k$. For $1 \le x \le n$:  $R_0 \le R_{x} \le R_{n+x} \le \ldots \le R_{(p-1)n + x}$ in the weak order of containment.  
Label the $1+np$ regions of the arrangement iteratively as follows. First of all, the
$n$ hyperplanes $h_1, \ldots h_n$  are named so that each $h_x$ forms an edge of the $n$-gon that shares vertices with the edges formed by $h_{\bar{x} \pm 1}$, and so that $h_2$ follows $h_1$ moving along the edges in a counter-clockwise direction; the central region is labeled by $R_0$;
label the remaining regions in the following iterative process: for each $k\leq n$ and $m\leq p-1$, choose $R'$ and name it by
$R_{(m+1)n+k}$ such that $sep(R_0, R_{m+( (k+1) mod(n))}) = sep(R_{mn+k}, R')$. Repeat this process till $k$ reaches $n$. Here and throughout this section $\bar{i}$ denotes $(i \, {\rm mod} \, n)$.

%
\begin{figure}
\includegraphics[scale=.7, trim={0in .02in 0in .02in},clip]{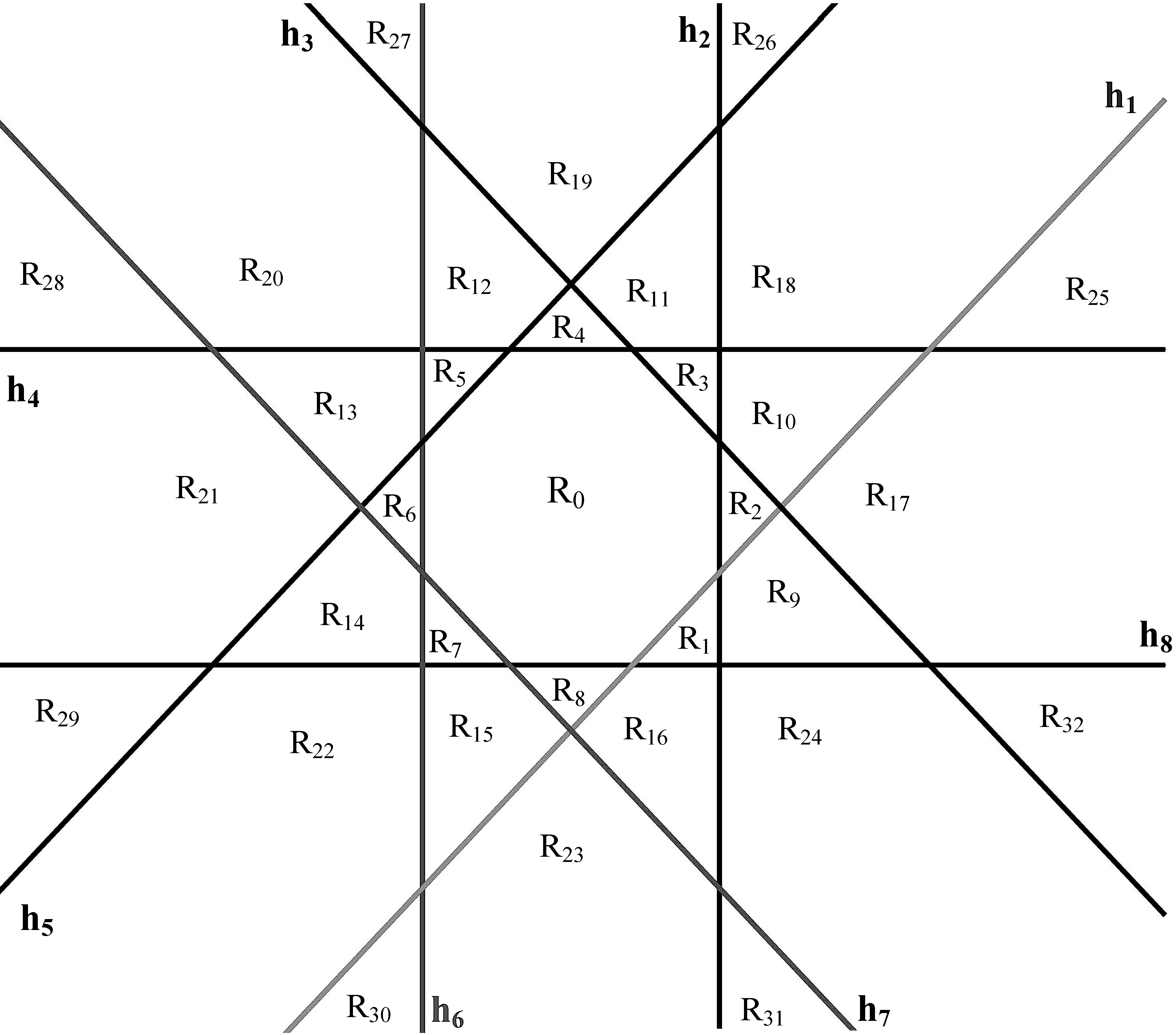}
\caption{The labelled $C_{8}$ hyperplane arrangement.}
\end{figure}
For $1 \le x \le n$ and $1 \le r \le p$, $\#sep(R_{0}, R_{(r-1)n+x}) = r$, so
	$Q_{r}=\bq_n^r$. 
Also for any $\bar{a}\equiv \overline{b+1}$ we have that
\begin{equation*}
\#sep\left(R_{(i-1)n+1}, R_{(j-1)n+x}\right)=\#sep\left(R_{(i-1)n+\bar{a}}, R_{(j-1)n+\bar{b}}\right)
\end{equation*}
For any integer $1 \le a \le n$, the $(x', x)$-entry of $C_{ij}$ is equal to the $(\overline{x'+a}, \overline{x+a})$-entry, where the overline
refers to modulo $n$. 
Therefore, $C_{ij}$ is a circulant matrix.

Suppose $n \ge 3$ and $i \le j$.
For $1 \le x \le n$, note that
$C_{1j}$ is given by $\#sep\left(R_1, R_{((j-1)n+x}\right)$, $C_{2j}$ is given by $\#sep\left(R_{n+1} R_{(j-1)n+x}\right)$, and $C_{3j}$ is given by $\#sep\left(R_{2n+1}, R_{(j-1)n+x}\right)$ etc.

Now consider $C_{ij}$ where $i \le j$.  Specifically for $i \le j \le p$ and $1 \le x \le n$,
\begin{equation*}
	\#sep\left(R_{(i-1)n+1}, R_{(j-1)n+x}\right) =\begin{cases}
		j-i, & sep\left(R_0, R_{(i-1)n+1}\right) \subseteq sep\left(R_0, R_{(j-1)n+x}\right),\\
 		j-i + 2k, & k \text{ elements from } sep\left(R_0, R_{(i-1)n+1}\right) \notin sep\left(R_0, R_{(j-1)n+x}\right),\\
 		j+i, & sep\left(R_0, R_{(i-1)n+1}\right) \cap sep\left(R_0, R_{(j-1)n+x}\right) = \emptyset. \end{cases}
\end{equation*}
%
%

%

By analyzing three cases we have that for $1 \le k < i$ and $i \le j$,
\begin{equation} \label{Cn Eij cases}
	\#sep\left(R_{(i-1)n+1}, R_{(j-1)n+x}\right) =\begin{cases}
		j-i, & x \in  \{\underbrace{1, n, n-1, \ldots, n-(j-i-1)}_{\text{$j-i+1$ terms in set}}\},\\
		j-i + 2k, & x \in \{k+1, n-(j-i-1)-k\},\\
		j+i, & x \in \{\underbrace{i+1, i+2, \dots, n-(j-1)}_{\text{$n+1-i-j$ terms in set}}\}.
	 \end{cases}
\end{equation}
Therefore $C_{ij}$ is indeed as given in \eqref{Cn Eij form}, and the theorem is proved.
\end{proof}
\begin{thm} \label{cn snf} For any $n \ge 3$, the Smith normal form of the $q$-Varchenko matrix for the cyclic model arrangement is
\setlength{\arraycolsep}{4pt}
\setlength{\extrarowheight}{-2pt}
\begin{equation}
 \label{snf cn}
\left( \begin{array}{ccc}1 & 0 & 0 \\ 0 & (1-q^2)I_{n} & 0 \\ 0 & 0 & (1-q^2)^2I_{n(p-1)}\end{array}\right).
 \end{equation}

The corresponding left and right transition matrices are (respectively) $SP$ and $(SP)^{t}$ where
\setlength{\arraycolsep}{2.75pt}
\setlength{\extrarowheight}{-2pt}
  \begin{equation} \label{transition matrix cn} \setlength{\extrarowheight}{-3pt}
	  P=\left(\begin{array}{ccccc}
	1 & 0 & 0 & 0 & 0\\
	-qI_{n \times 1} & I_n & 0 & 0 & 0\\
	0 & -qI_n & I_n & 0 & 0\\
	\vdots & \vdots & \ddots & \ddots  & \vdots \\
	0 & \dots & 0 & -qI_n & I_n
	\end{array}\right), \
	S= \left(\begin{array}{ccccc}
	1 & 0 & 0 & \dots & 0\\
	0 & I_n & 0 & \dots & 0\\
	0 & -qJ & I_n &  \dots & 0\\
	\vdots & \vdots & \ddots & \ddots & \vdots \\
	0 & \dots & 0 & -qJ & I_n
	\end{array}\right),
\end{equation}
where $J$ is the $n\times n$ permutation matrix $\begin{pmatrix} 0 & I_{n-1}\\ 1 & 0\end{pmatrix}$.
\end{thm}

\begin{proof}
Note that the first row and column of the $q$-Varchenko matrix $V_q$ (see \eqref{e:qVar}) 
%
 for $C_{n}$ can be transformed into its Smith normal form over $\Z[q]$ in two actions: congruent action by $P$ (i.e. $V\longrightarrow PVP^t$) followed by congruent action by $S$.
After left multiplication by $P$, the $Q^{t}_{k}$ blocks in $V_q$ become zeros since for $1 \le k \le p$, $Q^{t}_{k} - qQ^{t}_{k-1} = 0$.  Also the $(i, j)$-entry $C_{ij}$ in $V_q$ becomes
\begin{equation} \label{cn after step 1}
        C_{ij} - qC_{i-1, \ j} = \begin{cases}
        \displaystyle q^{j-i}(1-q^2)\left( \sum_{k=0}^{j-i}\left(J^{t}\right)^{k} +  \sum_{m=1}^{i-1}q^{2m}J^{m}\right),  & \mbox{ for }1\le i \le j,\\
      &\\
        \displaystyle  q^{j-i}(1-q^2)J^{j-i}\sum_{k=0}^{i-1}q^{2k}J^{k}, & \mbox{ for } i > j . \end{cases}
\end{equation}

Then after right multiplication by $P^{t}$, the $Q_{k}$ blocks become zeros since for $1 \le k \le p$, $Q_{k} - qQ_{k-1} = 0$.  Furthermore, the $(i, j)$-entry $C_{ij}$ block matrices in $V_q$ become
\begin{equation}  \label{cn after step 2}
\left(C_{ij} - qC_{i-1, j}\right) - q \left(C_{i, j-1} - q C_{i-1, j-1} \right) = (1-q^2) q^{|i-j|} J^{\overline{i-j}},
\end{equation}
which can be verified case by case by using \eqref{cn after step 1}.

Therefore, after congruent action of $P$
  the $q$-Varchenko matrix is transformed to:
\begin{equation}\label{congruence of P}
	P V_{q} P^{t} =
	 \left(\begin{array}{c cc cc}1 & 0 &  0 &  \dots & 0\\
	0 & (1-q^2)I_{n} & q(1-q^2)J^t &  \dots  & q^{p-1}(1-q^2)(J^t)^{p-1} \\
	0 & q(1-q^2)J & (1-q^2)I_{n} &  \dots & q^{p-2}(1-q^2)(J^t)^{p-2}\\
	\vdots & \vdots & \vdots  & \ddots & \vdots\\
	0 & q^{p-1}(1-q^2)J^{p-1} & q^{p-2}(1-q^2)J^{p-2} & \dots &(1-q^2) I_{n}
	\end{array}\right)_{(1+np) \times (1+np)}.
\end{equation}
Then, left multiplication by $S$ to $PV_qP^t$ leaves the first row unchanged.  The block matrix entries below the diagonal become zero, the block matrix entries along the diagonal become $(1-q^2)^2I$ and the block matrix entries above the diagonal become $q^{j-i}(1-q^2)^2(J^{t})^{j-i}$.

Multiplication on the right by $S^{t}$ to the resulting matrix leaves the entries on and below the diagonal unchanged, and the entries above the diagonal become zeros.
Indeed for $1 < i < j$,
 \begin{flalign}\label{cn step 4}
        &	C_{ij}-qC_{i-1, j} - q \left(C_{i, j-1} - qC_{i-1, j-1}\right) - qJ\Big(  C_{i-1, j} - qC_{i-2, j} - q \left(C_{i-1, j-1} - qC_{i-2, j-1}\right) \Big)- \notag \\
        &	 \bigg( C_{i, j-1} - qC_{i-1, j-1} - q\left(C_{i, j-2} - q C_{i-1, j-2}\right) - qJ \Big(C_{i-1, j-1} - qC_{i-2, j-1} - q \left(C_{i-1, j-2} - q C_{i-2, j-2} \right) \Big) \bigg) (qJ^t)\notag\\
        &	 = q^{j-i} (1-q^2) (J^t)^{j-i} -q^{j-1-i} (J^t)^{j-1-i} qJ^t= 0.
\end{flalign}
Thus the congruent action by $S$ to the matrix in (\ref{congruence of P}) transforms the $q$-Varchenko matrix $V_q$ to its Smith normal form in (\ref{snf cn})
\end{proof}
\subsection{The Dihedral Model $D_n$} Now we take the hyperplanes from the cyclic model and move them into $\R^3$ so that the lines forming edges of the regular $n$-gon become planes forming a regular $n$-gon shaped cylinder. Then we add one hyperplane $h$ that is perpendicular to $h_1, h_2, \ldots, h_n$. This doubles the number of regions. The distance enumerator for $\mathcal{A}$ (with respect to $R^{+}_{0}$) is
\begin{equation*}D_{\mathcal{A}, R^{+}_{0}}(t) = 1+ t + \sum_{k=1}^{p} n (t^k +t^{k+1}), \ \mbox{ where } \ p = \lfloor \frac{n+1}{2} \rfloor.
\end{equation*}

We will refer to the regions above $h$ as $R_{x}^+$ and the regions below $h$ as $R_{x'}^-$. The $R_{x}^+$ regions are labelled according to the cyclic model.  Effectively, we label the $R_{x}^{-}$ regions by taking a reflection of the cyclic model regions along the axis of symmetry between $R^{+}_1$ and $R^{+}_n$.
\begin{thm} \label{Dn Vq}
The $q$-Varchenko matrix for the dihedral model arrangement is a $2(np+1) \times 2(np+1)$ matrix with the following block form:
\setlength{\arraycolsep}{3pt}
\setlength{\extrarowheight}{-1pt}
\begin{equation}
V_q=\left(\begin{array}{cc ccc}1 & \overline{Q}_1 & \overline{Q}_2 & \dots & \overline{Q}_p\\
	\overline{Q}^{t}_{1} & \overline{C}_{11} & \overline{C}_{12} &  \dots & \overline{C}_{1p} \\
	\overline{Q}^{t}_{2} & \overline{C}_{12} & \overline{C}_{22}  & \dots  & \overline{C}_{2p}\\
	\vdots & \vdots & \vdots & \ddots & \vdots\\
	\overline{Q}^{t}_{p} & \overline{C}_{1p} & \overline{C}_{2p} & \dots & \overline{C}_{pp}
	\end{array}\right),
\end{equation}
where the blocks are defined by
\begin{equation*}
\overline{Q}_{i}=\begin{pmatrix} Q_{i} & qQ_{i}\\ qQ_{i} & Q_{i}\end{pmatrix}, 0\le i\le p; \qquad
\overline{C}_{ij}=\begin{pmatrix} C_{ij} & qJ^{i-1}KC_{ij}\\ qJ^{i-1}KC_{ij} & C_{ij}\end{pmatrix}, 1\le i\le j\le p.
\end{equation*}
Here $Q_0=1, Q_{r}=\bq_n^r$, $C_{ij}$ is defined in \eqref{Cn Eij form}, $K$ is the $n\times n$ skew diagonal matrix
or $RC(0, \cdots, 1)_{n\times n}$
and $J=C(0, 1, 0, \cdots, 0)_{n\times n}$.
%
\end{thm}
\begin{proof}For $1 \le k \le p$, there are $n$ regions $R^{\pm}$ such that $\#sep\left(R_{0}^{\pm}, R^{\pm}\right) = k$.  For $1 \le x \le n$:  $R_{0}^{\pm} \le R_{x}^{\pm} \le R_{n+x}^{\pm} \le \ldots \le R_{(p-1)n + x}^{\pm}$ in the weak order of containment. 

Label the $2+2np$ regions of the arrangement in the following way:
\begin{enumerate}
\item $h$ is the hyperplane perpendicular to the hyperplanes  $h_{1}, \ldots, h_{n}$;
\item For $1 \leq x \leq n$, the hyperplane $h_x$ forms an edge of the $n$-gon that shares vertices with the edges formed by $h_{x \pm 1 \ (mod \ n)}$ so that $h_2$ follows $h_1$ moving along the edges in a counter-clockwise direction;
\item the central regions above $h$ is labeled as $R^{+}_{0}$ and below $h$ is labeled as $R^{-}_{0}$;
\item for $1 \le x \le n$, label the region $R$  as $R_x$ if $sep\left(R^{+}_0, R\right) = h_x$  and as $R^{-}_{n+1-x}$ if $sep\left(R^{-}_0, R\right) = h_{x}$; 
\item  For $1 \leq k, m \leq n$, $m+1 \leq p-1$, label the region $R$ as $R^{+}_{(m+1)n+k}$ if
$sep\left(R^{+}_{0}, R^{+}_{m+( (k+1) \ mod(n))}\right) = sep\left(R^{+}_{mn+k}, R'^{+}\right)$
and label it $R^{-}_{(m+1)n+k}$ if
$ sep\left(R^{-}_{0}, R^{-}_{m+( (k+1) \ mod(n))}\right)=sep\left(R^{-}_{mn+k}, R'^{-}\right)$.
\end{enumerate}
Now for $1 \le x \le n$ and $1 \le r \le p$, it is easy to see that $\#sep(R_{0}^{+}, R^{+}_{(r-1)n+x}) = r$ and $\#sep(R_{0}^{+}, R^{-}_{(r-1)n+x}) = r+1$.  Therefore $Q_{r} = (q^{r}, \ldots, q^{r})_{1 \times n}=\bq_n^r$.

Note that though the process is identical for the $R^+$ and $R^-$ regions, the hyperplane involved in each step is not.
For example,
\begin{flalign*}
&sep(R_{1}^{+}, R_{n+1}^{+}) = \{h_{2}\}, \ sep(R_{1}^{-}, R_{n+1}^{-}) = \{h_{n-1}\}; \ \
sep(R_{n+1}^{+}, R_{2n+1}^{+}) = \{h_{3}\}, \ sep(R_{n+1}^{-}, R_{2n+1}^{-}) = \{h_{n-2}\}.
\end{flalign*}
However,
$sep\left(R_0^{\pm}, R_1^{\pm}\right) \subseteq sep\left(R_0^{\pm}, R_{n+1}^{\pm}\right) \subseteq sep\left(R_{0}^{\pm}, R_{2n+1}^{\pm}\right) \subseteq \ldots \subseteq sep(R_0^{\pm}, R_{(p-1)n + 1}^{\pm}).$

For $1 \le i, j \le p$,
\begin{flalign}
	& sep\left(R_{0}^+, R_{(i-1)n+x}^+\right)=\{h_{x}, h_{x+1}, h_{x+2}, \ldots, h_{x+(i-1)}\},  \\
	& sep\left(R_{0}^-, R_{(j-1)n+x}^-\right)=\{h_{n+1-x}, h_{n+1-(x+1)}, h_{n+1-(x+2)}, \ldots, h_{n+1-(x+(j-1))}\},
\end{flalign}
where the index $i$ in $h_i$ is modulo $n$.
\begin{figure}
\includegraphics[scale=.3, trim={.1in .1in 0in .05in},clip]{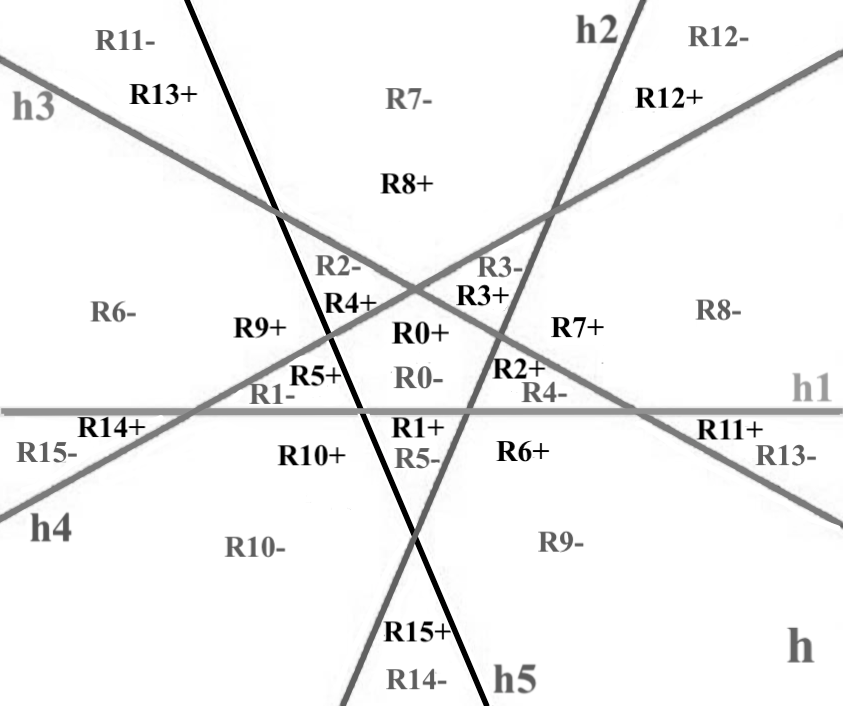}
\caption{The labelled $D_{5}$ arrangement, where $k=5$ and $m=2$. The hyperplane $h$ is this page. The $+$ regions are on the front side of the page; the $-$ regions are on the back side of the page.}
\end{figure}
Hence the $2(np+1) \times 2(np+1)$ $q$-Varchenko matrix $V_q$ for $D_n$ is:

\begin{equation}
\setlength{\arraycolsep}{6pt}
\setlength{\extrarowheight}{-2pt}
V_{q} =	\left(\begin{array}{cc |cc |c  |cc |c |cc}1 & q & Q_1 & Q_2  & \ldots & Q_r & Q_{r+1} & \ldots & Q_p & Q_{p+1}\\
        q & 1 & Q_2 & Q_1 &  \ldots & Q_{r+1} & Q_r & \ldots & Q_{p+1} & Q_p\\
        \hline
        Q_1^{t} & Q_{2}^{t} & C_{1^+1^+} & C_{1^+1^-} &  \ldots & C_{1^+r^+} & C_{1^+r^-} & \ldots & C_{1^+p^+} & C_{1^+p^-} \\
        Q_{2}^{t} & Q_{1}^{t} & C_{1^-1^+} & C_{1^-1^-}  & \ldots & C_{1^-r^+} & C_{1^-r^-} & \ldots & C_{1^-p^+} & C_{1^-p^-} \\
           \hline
        \vdots & \vdots & \vdots & \vdots  & \ddots & \vdots & \vdots & \dots & \vdots & \vdots \\
                   \hline
        Q_r^{t} & Q_{r+1}^{t} & C_{r^+1^+} & C_{r^+1^-}  & \ldots & C_{r^+r^+} & C_{r^+r^-} & \ldots & C_{r^+p^+} & C_{r^+p^-}\\
        Q_{r+1}^{t} & Q_r^{t} & C_{r^-1^+} & C_{r^-1^-} & \ldots & C_{r^-r^+} & C_{r^-r^-} & \ldots & C_{r^-p^+} & C_{r^-p^-} \\
                   \hline
                \vdots & \vdots & \vdots & \vdots & \dots & \vdots & \vdots & \ddots & \vdots & \vdots \\
                   \hline
        Q_p^{t} & Q_{p+1}^{t} & C_{p^+1^+} & C_{p^+1^-} & \ldots & C_{p^+r^+} & C_{p^+r^-} & \ldots & C_{p^+p^+} & C_{p^+p^-}\\
        Q_{p+1}^{t} & Q_p^{t} & C_{p^-1^+} & C_{p^-1^-} &  \ldots & C_{p^-r^+} & C_{p^-r^-} & \ldots & C_{p^-p^+} & C_{p^-p^-} \end{array} \right),
\end{equation}
where $Q_{r} =\bq_n^r$, $C_{i^{\pm}j^{\pm}} = C_{i^{\mp}j^{\mp}}$, $C_{j^{\pm}i^{\pm}} = C_{i^{\pm}j^{\pm}}^t$, $C_{i^{\pm}j^{\mp}} = C_{i^{\mp}j^{\pm}}$ and $C_{j^{\mp}i^{\pm}} = C_{i^{\pm}j^{\mp}}^t.$
As immediate consequences of the indexing method of the regions, $C_{i^{\pm}j^{\pm}}$ are circulant matrices, and $C_{i^{\pm}j^{\pm}}= C_{ij}$ as stated in Theorem \ref{Cn Vq}.

For $1 \le x \le n$ and $0 \leq a \le n$
\begin{equation}\label{Dn Eij ++ relations}
\#sep\left(R^{\pm}_{(i-1)n+1+a}, R^{\pm}_{(j-1)n+x-a}\right)=\#sep\left(R^{\pm}_{(i-1)n+1}, R^{\pm}_{(j-1)n+x}\right).
\end{equation}
which is a consequence of the fact that when $x+k=n$ for some $1 \le k < n$, $R^{\pm}_{(j-1)n + (x+k \ mod(n))} = R^{\pm}_{(j-1)n}$.

For $1 \le x', x \le n$, the $(x', x)$-entry of $C_{i^+ j^-}$ in $V_q$ is equal to the number of the hyperplanes separating the regions $R^{+}_{(i-1)n + x'}$ and $R^{-}_{(j-1)n+x}$. Now let's compute it as in the proof of Theorem \ref{Cn Vq}
starting with $\#sep(R_{(i-1)n +x'}, R_{(j-1)n+x})$ with $x'=n-(i-1)$. This will give us the $n-(i-1)$ row of $C_{i^+j^-}$ instead of the first row.  However, as a result of the indexing convention, we see that
 that the $C_{i^{\pm}j^{\mp}}$ blocks are reverse-circulant matrices.  Therefore they are determined entirely by the first row, or equivalently, any particular row.
For $1 \le x \le n$, we have the following observations:
\begin{align}
	sep\left(R_{0}^+, R_{(i-1)n+n-(i-1)}^+\right) & 
= \{h_n, h_{n-1}, h_{n-2}, \ldots, h_{n-(i-1)}\}, \\
	sep\left(R_{0}^+, R_{(j-1)n+x}^-\right)&=\{h, h_{n+1-x}, h_{n+1-(x+1)}, h_{n+1-(x+2)}, \ldots, h_{n+1-(x+(j-1))}\}.
\end{align}
Therefore by Proposition \ref{sep set},
\begin{multline}
	sep\left(R^{+}_{(i-1)n+n-(i-1)}, R^{-}_{(j-1)n + x}\right) =
	\{h_n, h_{n-1}, \ldots, h_{n-(i-1)}\} \cup \{h, h_{n+1-x}, \ldots, h_{n+1-(x+(j-1))}\} -\\
	 ( \{h_n, h_{n-1}, \ldots, h_{n-(i-1)}\} \cap \{h, h_{n+1-x}, \ldots, h_{n+1-(x+(j-1))}\}).
\end{multline}
Then we have for $1 \le k < i\le j$,
%
%
\begin{equation}
	\#sep\left(R^+_{(i-1)n+n-(i-1)}, R^-_{(j-1)n+x}\right) =\begin{cases}
	 j+1-i, & x \in \{\underbrace{1, n, n-1, \ldots, n-(j-i-1)}_{\text{$j-i+1$ terms in set}}\}.\\
	 j+1-i + 2k, & x \in \{k+1, n-(j-i-1)-k\}.\\
	 j+1+i, & x \in \{\underbrace{i+1, i+2, \dots, n-(j-1)}_{\text{$n+1-i-j$ terms in set}}\}.
	 \end{cases}
\end{equation}
Comparing it to \eqref{Cn Eij cases}, we see that this is precisely $1+\#sep(R_{(i-1)n+1}^{+}, R_{(j-1)n+x}^{+})$.  Therefore, the $n-(i-1)$ row of $C_{i^+j^-}$ equals $q$ times the first row of $C_{ij}$ as defined in Theorem \ref{Cn Vq}.  Left multiplication by $J^{i-1}K$ acting on the circulant matrix $C_{ij}$ takes the first row to the $n$th row, and then takes the $n$th row to the $n-(i-1)$ row. Since both $C_{i^{+}j^{-}}$ and $qJ^{i-1}KC_{ij}$ are reverse circulant matrices, the equality of one row of each matrix is sufficient to determine the equality of the two matrices.  Therefore, $qJ^{i-1}KC_{ij} = C_{i^+j^-}$.
 \end{proof}
\begin{proposition} \label{C_{ij}^t=J^{j-i}C_{ij}} For $i \le j$,  $C_{i j}^t = J^{j-i} C_{i j}.$
\end{proposition}
\begin{proof}
Note that for $1 \leq r, c \leq n$ we have  $J_{r c} = \delta_{r+1 \ (mod \ n), c}$, and hence $(J^{j-i})_{r c} = \delta_{r + j - i, c}$. So
\begin{equation*}
	(J^{j-i} C_{i j})_{r c} = \sum_{k=1}^n (J^{j-i})_{rk} (C_{ij})_{kc} = \sum_{k=1}^n \delta_{r+j-i, k} (C_{i j})_{k c} = (C_{i j})_{r+j-i, c}.
\end{equation*}
Since it is a circulant matrix, $J^{j-i} C_{ij}$ can be written using the $(1+j-i)^{th}$ row of $C_{i j}$. Therefore, we have
\begin{equation*}
	J^{j-i} C_{i j} = C(\underbrace{q^{j-i}, \ldots, q^{j-i},}_{\text{$j-i+1$}} q^{j-i+2}, \ldots, q^{j-i+2(i-1)}, \underbrace{q^{j+i}, \ldots, q^{j+i},}_{\text{$n+1-i-j$}} q^{j-i+2(i-1)}, \ldots, q^{j-i+2}) = C_{ij}^t.
\end{equation*}
\end{proof}
 \begin{thm}  \label{dn snf} For $n \ge 3$, the Smith normal form of the $q$-Varchenko matrix $V_q$ given in Theorem \ref{Dn Vq} for
 the dihedral model arrangement is
 \setlength{\arraycolsep}{3pt}
 \setlength{\extrarowheight}{-2pt}
\begin{equation}
         \left( \begin{array} {cccc}
        1 & 0 & 0 & 0\\
        0 & (1-q^2)I_{n+1} & 0 & 0\\
        0 & 0 & (1-q^2)^2 I_{np} & 0\\
        0 & 0 & 0 & (1-q^2)^3 I_{n(p-1)}
        \end{array} \right).
\end{equation}

This is obtained by the congruent action of $R T S P$, where
\begin{equation*}
\setlength{\arraycolsep}{3pt}
\setlength{\extrarowheight}{-2pt}
P=
\left(\begin{array}{cc|cc|c|cc|cc}
1 & 0 & 0 & 0  &  \dots  & 0 & 0 & 0 & 0   \\
-q & 1 & 0 & 0 &  \dots  & 0 & 0    & 0 & 0 \\
\hline
-q1 & 0 & I_n & 0 & \dots  & 0 & 0  & 0 & 0  \\
0 & -q1 & 0 & I_n &  \dots & 0 & 0  & 0 &  0  \\
\hline
0 & 0 & -qI_n & 0 &  \ddots  & 0 & 0 & 0 & 0  \\
0 & 0 & 0 & -qI_n & \ddots  & 0  & 0  & 0 & 0 \\
\hline
\vdots & \vdots & \vdots & \vdots & \ddots & \vdots &  \vdots & \vdots & \vdots  \\
\hline
 0 & 0 & 0 & 0 &  \dots  & -qI_n & 0 & I_n & 0\\
0 & 0 & 0 & 0 &  \dots   & 0 & -qI_n & 0 & I_n
\end{array}\right), \
\setlength{\arraycolsep}{3pt}
\setlength{\extrarowheight}{-2pt}
S =
\left(\begin{array}{cc|cc|c|cc|cc}
1 & 0 & 0 & 0  &  \dots  & 0 & 0 & 0 & 0   \\
0 & 1 & 0 & 0 &  \dots  & 0 & 0    & 0 & 0 \\
\hline
0 & 0 & I_n & 0 & \dots  & 0 & 0  & 0 & 0  \\
0 & 0 & 0 & I_n &  \dots & 0 & 0  & 0 &  0  \\
\hline
0 & 0 & -qJ & 0 &  \ddots  & 0 & 0 & 0 & 0  \\
0 & 0 & 0 & -qJ & \ddots  & 0  & 0  & 0 & 0 \\
\hline
\vdots & \vdots & \vdots & \vdots & \ddots & \vdots &  \vdots & \vdots & \vdots  \\
\hline
 0 & 0 & 0 & 0 &  \dots  & -qJ & 0 & I_n & 0\\
0 & 0 & 0 & 0 &  \dots   & 0 & -qJ & 0 & I_n
        \end{array}\right),
\end{equation*}
\begin{equation*}
\setlength{\arraycolsep}{3pt}
\setlength{\extrarowheight}{-2pt}
  T=
        \left(\begin{array}{cc|cc|cc|c|cc}
        1 & 0 & 0 & 0  & 0 & 0 & \dots  & 0 & 0\\
        0 & 1 & 0 & 0 & 0 & 0 & \dots  & 0 & 0\\
        \hline
        0 & 0 & I_n & 0 & 0 & 0 & \dots & 0 & 0\\
        0 & 0 & -qK & I_n & 0 & 0 & \dots  & 0 &  0\\
       \hline
	 0 & 0 & 0 & 0 & I_n & 0 & \dots & 0 & 0\\
        0 & 0 & 0 & 0 & -qJK & I_n & \dots  & 0 &  0\\
       \hline
        \vdots & \vdots & \vdots &\vdots & \vdots & \vdots &  \ddots & \vdots & \vdots \\
        \hline
        0 & 0 & 0 & 0 & 0 & 0 & \dots  & I_n & 0\\
        0 & 0 & 0 & 0 & 0 & 0 & \dots  & -qJ^{p-1}K & I_n
        \end{array}\right),
\end{equation*}
\begin{equation*}
\setlength{\arraycolsep}{3pt}
\setlength{\extrarowheight}{-4pt}
	R = \left( \begin{array}{cc cc cc} A_1 & 0 & 0 & \dots  & 0 & 0\\ 0 & A_{2} &  &  & \dots &  0\\
	0 & 0 & \ddots & \dots & \dots & 0\\ 0 & 0  & \dots & A_{j} &   & 0\\
	\vdots & \vdots & \vdots & \vdots & \ddots & \vdots \\ 0 & \dots  & 0 &  &  & A_{p}\end{array} \right); \ \mbox{ where }
	A_1= \left(\begin{array}{cccc}
	1 & 0 & 0 & 0\\
	0 & 1 & 0 & 0\\
	0 & 0 & I_n & 0\\
	0 & 0 & 0 & I_n
	\end{array}\right), \ A_{j} = \left( \begin{array}{cccc} I_{n} & 0 & \ldots & 0\\ 0 & \ldots & 0 & I_{n}\end{array} \right)_{2n \times np} \mbox{for $2 \le j \le p$}.
\end{equation*}
\setlength{\arraycolsep}{6pt}
\setlength{\extrarowheight}{0pt}
 \end{thm}
\begin{proof} We will index the first row (resp. column) of the Varchenko matrix $V_q$ with entries $Q_k$ (resp. $Q_k^t$) as row $i=0$ (resp. column $j=0$). The entry $C_{11}$ will be indexed by row $i=1$ and column $j=1$. The Varchenko matrix
$V_{q}$ for $D_{n}$ can be transformed into its Smith normal form over $\Z[q]$ in successive congruent actions
of $P$, $S$, $T$ and $R$.

After left multiplication by $P$, the $Q^{t}_{k}$ blocks become zeros since for $1 \le k \le p$, $Q^{t}_{k} - qQ^{t}_{k-1} = 0$. Furthermore, the original $C_{ij}$ block matrices have the form given in \eqref{cn after step 1} and  the original $qJ^{i-1}KC_{ij}$ block matrices have the form
\begin{equation} \label{dn lemmas summary}
  \begin{cases}
\displaystyle q^{j}(1-q^2)\left(\sum_{k=0}^{j-1} J^{k}\right)K, & \mbox{for } i = 1,\\
& \\
\displaystyle q^{j-i+1}(1-q^2)J^{j-1}\left(\sum_{k=0}^{j-i}(J^{t})^{k} + \sum_{m=1}^{i-1}q^{2m}J^{m}\right)K, & \mbox{for } i>1.
\end{cases}
\end{equation}

After right multiplication by $P^{t}$, the $Q_{k}$ blocks become zeros since for $1 \le k \le p$, $Q_{k} - qQ_{k-1} = 0$ and the first column remains unchanged.
 The block matrices with the original entry $C_{ij}$ become
        \begin{equation} \label{cn in dn after step 2}
        (1-q^2)q^{|i-j|}J^{\overline{i-j}}.
        \end{equation}
Furthermore, the original $qJ^{i-1}KC_{ij}$ block entries become
\begin{equation} \label{dn step 2 cor sum}
	q^{j-i+1} (1-q^2) J^{j-1}\left(  (1-q^2)\left( \sum_{k=0}^{i-2} q^{2k} J^{k} \right)+ q^{2(i-1)} J^{i-1} \right) K.
\end{equation}
This can be verified by direct computations  case  by case, depending on the position of the original $qJ^{i-1}KC_{ij}$ blocks.

Then, left multiplication by the matrix  $S$ has the effect of beginning with the last two rows and subtracting $qJ$ times the previous two rows; then the same operation is repeated by subtracting $qJ$ times the third-to-last two rows from the second-to-last two rows, and so forth. The $qKC_{1 j}$-entries remain the same: $q^{j} (1-q^2) J^{j-1}K$.

For all of the $qJ^{i-1}KC_{i j}$-entries that are below the diagonal and for which $i \neq j$, left multiplication by $S$ carries out the operation of subtracting $qJ$ times the $qJ^{i-1}KC_{i, j-1}$-entry from the $qJ^{i-1}KC_{i j}$-entry. In this case, the result equals $0$.  For all of the $qJ^{i-1}KC_{i j}$-entries that are either on or above the diagonal or for which $i = j$, left multiplication by $S$ carries out the operation of subtracting $qJ$ times the $qJ^{i-2}KC_{i-1, j}$-entry from the $qJ^{i-1}KC_{i j}$-entry.


Right multiplication by $S^{t}$ has the effect of beginning with the last two columns and subtracting the previous two columns times $qJ^t$; then the same operation is repeated by subtracting the third-to-last two columns times $qJ^t$ from the second-to-last two columns, and so forth.  The $qKC_{11}$-entries remain the same throughout this step, and the $qJ^{i-1}KC_{i i}$-entries remain $q(1-q^2)^2J^{i-1} K$, since the corresponding entries to the left of them are zeros.
This operation gives us zeros above the block diagonal entries for which $i \neq j$, since  for $ i > j$, the current $qJ^{i-1}KC_{i j}$-entry minus the current $qJ^{i-1} K C_{i, j-1}$-entry times $qJ^t$ is
\begin{equation*}
	q^{j-i+1}(1-q^2)^2 J^{j-1}K - (q^{(j-1) -i+1} (1-q^2)^2 J^{j-2}K) qJ^t
	= q^{j-i+1} (1-q^2)^2 (J^{j-1}K - J^{j-2}K J^t) = 0.
\end{equation*}

Left multiplication by $T$ has the effect of beginning with the last two rows, subtracting $qJ^{p-1}K$ times the second to last row from the last row; then the same operation is repeated by subtracting $qJ^{p-2}K$ times the fourth to last row from the third to last row, and so on. The $qJ^{i-1}KC_{ii}$-entries above the diagonal remain the same, and the ones below the diagonal become zero since each $qJ^{i-1}KC_{ii}$-entry equals $qJ^{i-1}K$ times the corresponding $C_{ii}$-entry which is $(1-q^2)I$ for $C_{11}$ and $(1-q^2)^2 I$ for $i \ge 2$.

Along the diagonal, we get $(1-q^2)I$ in the first $\overline{C_{11}}$ diagonal position and $(1-q^2)^2I$ in its second diagonal position:
\begin{align*}
 	&(1-q^2)I - qK\left( q(1-q^2)K \right) = (1-q^2) I - q^2 (1-q^2) K^2 = (1-q^2)I - q^2(1-q^2)I = (1-q^2)^2I.
\end{align*}

For $i \ge 2$, we get $(1-q^2)^2I$ in the first $\overline{C_{ii}}$ diagonal position and $(1-q^2)^3$ in  its second diagonal position:
\begin{align*}
        &(1-q^2)^2 I - qJ^{i-1}K \left( q(1-q^2)^2J^{i-1} K \right) = (1-q^2)^2 I - q^2 (1-q^2)^2 J^{i-1} K J^{i-1}K\\
        & = (1-q^2)^2 I - q^2 (1-q^2)^2 I = (1-q^2)^3 I.
\end{align*}

At this point, we have calculated $TSP V_{q} (SP)^{t}$ which is:

\setlength{\extrarowheight}{-2pt}
  \setlength{\arraycolsep}{3pt}
 \begin{flalign*} 
&  \left(\begin{array}{cc|cc|cc|c|cc}
         1 & 0 & 0 &  0 & 0 & 0 & \dots &  0 & 0\\
        0 & t &  0 & 0 & 0 & 0 & \dots & 0 & 0\\
        \hline
        0 & 0 & tI_n & qtK & 0  & 0 & \dots & 0 & 0\\
        0 & 0 & 0 & t^2 I_n & 0  & 0 & \dots & 0 & 0\\
        \hline
        0 & 0 & 0 & 0 & t^2I_n & qt^2JK & \dots & 0   & 0\\
        0 & 0 & 0 & 0 & 0 & t^3I_n &  \dots &0  & 0\\
        \hline
        \vdots & \vdots & \vdots & \vdots & \vdots & \vdots &  \ddots & \vdots  & \vdots\\
        \hline
        0 & 0 &  0 & 0 & 0 & 0& \dots  &t^2I_n & qt^2J^{p-1}K\\
        0 & 0 &  0 & 0 & 0 &  0 & \dots  & 0 & t^3I_n
        \end{array}\right),
        \end{flalign*}
where $t=1-q^2$.

Right multiplication by $T^{t}$ has the effect of beginning with the last column and subtracting the previous column times $q(J^{p-1}K)^t = qJ^{p-1}K$; then the same operation is repeated by subtracting the fourth-to-last column times $qJ^{p-1}K$ from the third-to-last column, and so on.
The $qJ^{i-1}KC_{ii}$-entries above the diagonal become zero since each $qJ^{i-1}KC_{ii}$-entry equals $qJ^{i-1}K$ times the corresponding $(i, i)$-entry which is $(1-q^2)I$ for $C_{11}$ and $(1-q^2)^2 I$ for $i \ge 2$.  Now the matrix is in the diagonal form:

\begin{equation*}
\setlength{\extrarowheight}{-2.5pt}
          		  \setlength{\arraycolsep}{2pt}
		  TSP V_{q} (TSP)^{t} =
        \left(\begin{array}{cc|cc|cc|c|cc}
         1 & 0 & 0 & 0 & 0 & 0 &  \dots & 0 & 0\\
        0 & t & 0 & 0 & 0 & 0 & \dots & 0 & 0\\
        \hline
        0 & 0 & tI_n & 0 & 0  & 0 & \dots & 0 & 0\\
        0 & 0 & 0 & t^2 I_n & 0 & 0 & \dots & 0 & 0\\
        \hline
        0 & 0 & 0 & 0 & t^2I_n & 0 & \dots & 0 & 0\\
        0 & 0 & 0 & 0 & 0 & t^3I_n  & \dots & 0 & 0\\
        \hline
        \vdots & \vdots & \vdots & \vdots & \vdots & \vdots & \ddots & \vdots & \vdots \\
        \hline
        0 & 0 & 0  & 0 & 0 & 0 & \dots & t^2I_n & 0\\
        0 & 0 & 0  & 0 & 0 & 0 & \dots & 0 & t^3I_n
        \end{array}\right)
        \end{equation*}

Finally the left and right multiplication by the permutation matrices $R$ and $R^{t}$ (respectively) moves the $(1-q^2)^3I_{n}$ diagonal entries to the last $n(p-1)$ block diagonal places which gives the Smith normal form of the $q$-Varchenko matrix $V_q$ in this case.
\end{proof}
\begin{example} The $q$-Varchenko matrix for $D_{5}$ is the $32\times 32$ matrix
\begin{equation*}
\setlength{\extrarowheight}{0pt}
\setlength{\arraycolsep}{6pt}
\left(\begin{array}{cc|cc|cc|cc}1 & q & Q_1 & qQ_1 & Q_2 & qQ_2 & Q_3 & qQ_3\\
q & 1 & qQ_1 & Q_1 & qQ_2 & Q_2 & qQ_3 & Q_3\\
\hline
Q_1^{t} & qQ_{1}^{t} & C_{11} & qK C_{11} & C_{12} & qK C_{12} & C_{13} & qK C_{13}\\
qQ_{1}^{t} & Q_{1}^{t} & qK C_{11} & C_{11} & qK C_{12} & C_{12} & qK C_{13} & C_{13}\\
\hline
Q_2^{t} & qQ_2^{t} & C_{12}^{t} & qK C_{12} & C_{22} & q J K C_{22} & C_{23} & q J K C_{23}\\
qQ_2^{t} & Q_2^{t} & qK C_{12} & C_{12}^{t} & q J K C_{22} & C_{22}& q J K C_{23} & C_{23}\\
\hline
Q_3^{t} & qQ_{3}^t & C_{13}^{t} & qK C_{13} & C_{23}^{t} & q J K C_{23} & C_{33} & q J^2 K C_{33}\\
qQ_{3}^{t} & Q_{3}^{t} & qK C_{13} & C_{13}^{t} & q J K C_{23} & C_{23}^{t} & q J^2 K C_{33} & C_{33}
\end{array}\right).
\end{equation*}

The circulant matrices $C_{ij}$, $1\le j\le 3$ and the $1 \times 5$ row vectors $Q_{k}$, $k=1, 2, 3$, are as defined in Theorem \ref{Cn Vq}.
Here, $RC$ denotes reverse circulant matrices.
\begin{equation*}
	\begin{array}{ccc}
	qKC_{11} = RC(q^3, q^3, q^3, q^3, q), & qKC_{12} = RC(q^4, q^4, q^4, q^2, q^2), & qKC_{13} = RC(q^5, q^5, q^3, q^3, q^3),\\
	qJKC_{22} = RC(q^5, q^5, q^3, q, q^3), & qJKC_{23} = RC(q^6, q^4, q^2, q^2, q^4), &  qJ^2KC_{33} = RC(q^5, q^3, q, q^3, q^5).
	\end{array}
\end{equation*}	
\begin{equation*}
	\setlength{\arraycolsep}{4pt}
	\setlength{\extrarowheight}{-3pt}
 P =\left(\begin{array}{cc|cc|cc|cc}
1 & 0 & 0 & 0 & 0 & 0 & 0 & 0\\
-q & 1 & 0 & 0 & 0 & 0 & 0 & 0\\
\hline
-q 1 & 0 & I_5 & 0 & 0 & 0 & 0 & 0\\
0 & -q1 & 0 & I_5 & 0 & 0 & 0 & 0\\
\hline
0 & 0 & -qI_5 & 0 & I_5 & 0 & 0 & 0\\
0 & 0 & 0 & -qI_5 & 0 & I_5 & 0 & 0\\
\hline
0 & 0 & 0 & 0 & -qI_5 & 0 & I_5 & 0\\
0 & 0 & 0 & 0 & 0 & -qI_5 & 0 & I_5
\end{array}\right), \
        S=\left(\begin{array}{cc|cc|cc|cc}1 & 0 & 0 & 0 & 0 & 0 & 0 & 0\\
        0 & 1 & 0 & 0 & 0 & 0 & 0 & 0\\
        \hline
        0 & 0 & I_5 & 0 & 0 & 0 & 0 & 0\\
        0 & 0 & 0 & I_5 & 0 & 0 & 0 & 0\\
        \hline
          0 & 0 & -qJ & 0 & I_5 & 0 & 0 & 0\\
        0 & 0 & 0 & -qJ & 0 & I_5 & 0 & 0\\
         \hline
        0 & 0 & 0 & 0 & -qJ & 0 & I_5 & 0\\
        0 & 0 & 0 & 0 & 0 & -qJ & 0 & I_5
        \end{array}\right),
 \end{equation*}
\begin{equation*}
\setlength{\arraycolsep}{4pt}
\setlength{\extrarowheight}{-3pt}
  T=\left(\begin{array}{cc|cc|cc|cc}1 & 0 & 0 & 0 & 0 & 0 & 0 & 0\\
        0 & 1 & 0 & 0 & 0 & 0 & 0 & 0\\
        \hline
         0 & 0 & I_5 & 0 & 0 & 0 & 0 & 0\\
        0 & 0 & -qK & I_5 & 0 & 0 & 0 & 0\\
        \hline
       0 & 0 & 0 & 0 & I_5 & 0 & 0 & 0\\
        0 & 0 & 0 & 0 & -qJK & I_5 & 0 & 0\\
        \hline
        0 & 0 & 0 & 0 & 0 & 0 & I_5 & 0\\
        0 & 0 & 0 & 0 & 0 & 0 & -qJ^2K & I_5  \end{array}\right), \
        R=\left(\begin{array}{cccc|cccc}
        1 & 0 & 0 & 0 & 0 & 0 & 0 & 0\\
        0 & 1 & 0 & 0 & 0 & 0 & 0 & 0\\
        0 & 0 & I_5 & 0 & 0 & 0 & 0 & 0\\
        0 & 0 & 0 & I_5 & 0 & 0 & 0 & 0\\
        \hline
        0 & 0 & 0 & 0 & I_5 & 0 & 0 & 0\\
        0 & 0 & 0 & 0 & 0 & 0 & I_5 & 0\\
        \hline
        0 & 0 & 0 & 0 & 0 & I_5 & 0 & 0\\
        0 & 0 & 0 & 0 & 0 & 0 & 0 & I_5
        \end{array}\right).
        \end{equation*}
        \noindent
                 \begin{flalign*}
                   SP V_q P^{t}=
        &\setlength{\arraycolsep}{2.5pt}\setlength{\extrarowheight}{0pt}
 \left(\begin{array}{cc|cc|cc|cc}
         1 & 0 & 0 & 0 & 0 & 0 & 0 & 0\\
        0 & t & 0 & 0 & 0 & 0 & 0 & 0\\
        \hline
        0 & 0 & tI & qtK & qtJ^{4} & q^2tJK & q^2tJ^{3} & q^3tJ^2K\\
        0 & 0 & qtK & tI & q^2tJK & qtJ^{4} & q^3tJ^2K & q^2tJ^{3}\\
        \hline
        0 & 0 & 0 & 0 & t^2I & qt^2JK & qt^2 J^{4} & q^2t^2J^2K\\
        0 & 0 & 0 & 0 & qt^2JK & t^2I & q^2t^2J^2K & qt^2J^{4}\\
        \hline
        0 & 0 & 0 & 0 & 0 & 0 & t^2I & qt^2J^2K\\
        0 & 0 & 0 & 0 & 0 & 0 & qt^2J^2K & t^2I
        \end{array}\right).
        \end{flalign*}
where $t=1-q^2$. The Smith normal form is $\displaystyle RTSP V_{q} (RTSP)^{t} = diag(1, \ tI_{6}, \ t^2 I_{15}, \ t^3 I_{10})$.
\end{example}

\section{Polyhedral Models}\label{poly models}
There are five regular Platonic regular polyhedra whose symmetry groups are the exceptional Kleinian subgroups
\cite{Sl}. This section treats the exceptional ones. In each case we give the $q$-Varchenko matrix $V_q$, its Smith normal form
over $\mathbb{Z}[q]$ and its transition matrices. It can be checked by direct calculations that multiplying $V_q$ on the left and right by the given transition matrices we obtain the corresponding  Smith normal form. Further details on these calculations can be seen in \cite{Naomi}.

Denote the region $R_{i}$ as $i$; similarly, let $ij$ (respectively, $ijk$) denote the region $R$ separated from $R_{0}$ by the hyperplanes $h_{i}$ and $h_{j}$ (respectively, $h_{i}$, $h_{j}$ and $h_{k}$).
\subsection{Tetrahedron}
\label{chap-three}
 \begin{figure}[ht]
\includegraphics[scale=.15]{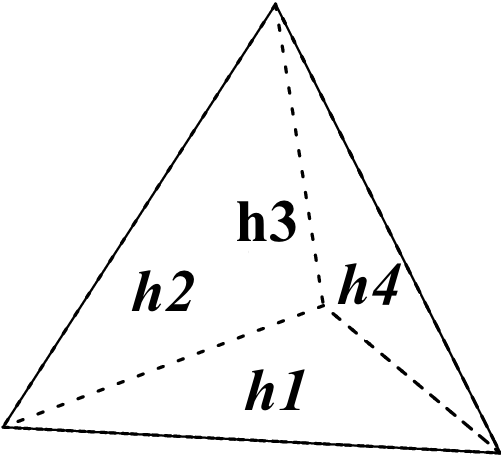}
\caption{Tetrahedron. Italicized hyperplanes are not visible from this view.}
\label{fig: tet}
\end{figure}
The distance enumerator with respect to $R_{0}$ is $D_{\mathcal{A}, R_0}(t) =1 + 4t + 6t^2 + 4t^3$.  Let $R_{0}$ denote the central region of the hyperplane arrangement.  Label the four regions $R_1, R_2, R_3, R_4$ so that $h_i$ separates $R_0$ from $R_i$ as shown in \fref{fig: tet}.

The $15$ regions are indexed in the following order, and denoted by the hyperplanes in $sep(R, R_0)$:
\begin{equation} \label{tet regions}
\setlength{\arraycolsep}{2.75pt}
\setlength{\extrarowheight}{0pt}
\begin{array}{c|| cccc|| cc cc |cc ||cc cc}
0 & 1 & 2 & 3 & 4 & 12 & 23 & 34 & 14 & 13 & 24 & 123 & 234 & 341 & 412
\end{array}
\end{equation}
\begin{equation*}
\setlength{\extrarowheight}{-2pt}
\setlength{\arraycolsep}{6pt}
V_q = \left( \begin{array}{c ||cccc ||cccc |cc|| cccc}
1 & q & q & q & q & q^2 & q^2 & q^2 & q^2 & q^2 & q^2 & q^3 & q^3 & q^3 & q^3\\
\hline
\hline
q & 1 & q^2 & q^2 & q^2 & q & q^3 & q^3 & q & q & q^3 & q^2 & q^4 & q^2 & q^2\\
q & q^2 & 1 & q^2 & q^2 & q & q & q^3 & q^3 & q^3 & q & q^2 & q^2 & q^4 & q^2\\
q & q^2 & q^2 & 1 & q^2 & q^3 & q & q & q^3 & q & q^3 & q^2 & q^2 & q^2 & q^4\\
q & q^2 & q^2 & q^2 & 1 & q^3 & q^3 & q & q & q^3 & q & q^4 & q^2 & q^2 & q^2\\
\hline
\hline
q^2 & q & q & q^3 & q^3 & 1 & q^2 & q^4 & q^2 & q^2 & q^2 & q & q^3 & q^3 & q\\
q^2 & q^3 & q & q & q^3 & q^2 & 1 & q^2 & q^4 & q^2 & q^2 & q & q & q^3 & q^3\\
q^2 & q^3 & q^3 & q & q & q^4 & q^2 & 1 & q^2 & q^2 & q^2 & q^3 & q & q & q^3\\
q^2 & q & q^3 & q^3 & q & q^2 & q^4 & q^2 & 1 & q^2 & q^2 & q^3 & q^3 & q & q\\
\hline
q^2 & q & q^3 & q & q^3 & q^2 & q^2 & q^2 & q^2 & 1 & q^4 & q & q^3 & q & q^3 \\
q^2 & q^3 & q & q^3 & q & q^2 & q^2 & q^2 & q^2 & q^4 & 1 & q^3 & q & q^3 & q\\
\hline
\hline
q^3 & q^2 & q^2 & q^2 & q^4 & q & q & q^3 & q^3 & q & q^3 & 1 & q^2 & q^2 & q^2\\
q^3 & q^4 & q^2 & q^2 & q^2 & q^3 & q & q & q^3 & q^3 & q & q^2 & 1 & q^2 & q^2\\
q^3 & q^2 & q^4 & q^2 & q^2 & q^3 & q^3 & q & q & q & q^3 & q^2 & q^2 & 1 & q^2\\
q^3 & q^2 & q^2 & q^4 & q^2 & q & q^3 & q^3 & q & q^3 & q & q^2 & q^2 & q^2 & 1\end{array} \right).
\end{equation*}

The Smith normal form is $diag\left(1, \ (1-q^2)I_{4}, \ (1-q^2)^2 I_{6}, \ (1-q^2)^3 I_{4} \right)$.

The left and right transformation matrices are (respectively) $P$ and $P^{t}$, where
\setlength{\arraycolsep}{3pt}
\setlength{\extrarowheight}{-1.5pt}
\begin{equation}
P =\left(\begin{array}{c||c||c|c||c}
1 & 0 & 0 & 0 & 0\\
\hline
\hline
-q1 & I_{4} & 0 & 0 & 0\\
\hline
\hline
q^21 & -q(I+J) & I_{4} & 0 & 0\\
\hline
q^21 & -q(I_{2} \ | \ I_{2}) & 0 & I_{2} & 0\\
\hline
\hline
-q^31 & q^2(I+J+J^2) & -q(I+J) & -q(I_{2} \ | \ I_{2})^{t} & I_{4}\end{array} \right),
\end{equation}
and
\begin{align*}
\left(I_{2} \ | \ I_{2}\right) = \left( \begin{array}{cc|cc} 1 & 0 & 1 & 0\\ 0 & 1 & 0 & 1\end{array}\right), \ \setlength{\extrarowheight}{-4pt}
J = \left(\begin{array}{cccc} 0 & 1 & 0 & 0\\ 0 & 0 & 1 & 0\\ 0 & 0 & 0 & 1\\ 1 & 0 & 0 & 0\end{array}\right).
\end{align*}
\subsection{Cube}
\label{chap-four}
Let $R_0$ denote the central region of the hyperplane arrangement
$\mathcal{A} = \{h_{1}, h_{2}, h_{3}, h_{4}, h_{5}, h_{6}\}$.  The distance enumerator with respect to $R_{0}$ is $D_{\mathcal{A}, R_0}(t) =1 + 6t + 12t^2 + 8t^3$.
Define each of the $6$ regions $R_{i}$ such that $sep(R_{i}, R_0) = h_{i}$ for $1 \le i \le 6$ where the hyperplanes $h_{i}$ are as shown in \fref{fig: cube corners}.

The $27$ regions are indexed in the following order, and denoted by the hyperplanes in $sep(R, R_0)$:
\begin{equation} \label{cube regions}
\setlength{\arraycolsep}{2.75pt}
\setlength{\extrarowheight}{0pt}
\begin{array}{c|| cccccc|| cccccc| cccccc|| cccccc| cc}
0 & 1 & 2 & 3 & 4 & 5 & 6 & 12 & 23 & 34 & 45 & 56 & 61 & 13 & 24 & 35 & 46 & 15 & 26 & 123 & 234 & 345 & 456 & 561 & 612 & 135 & 246 \end{array}
\end{equation}
\begin{figure}[ht]
\includegraphics[scale=.4, trim={0in 0.1in 0in 0.15in},clip]{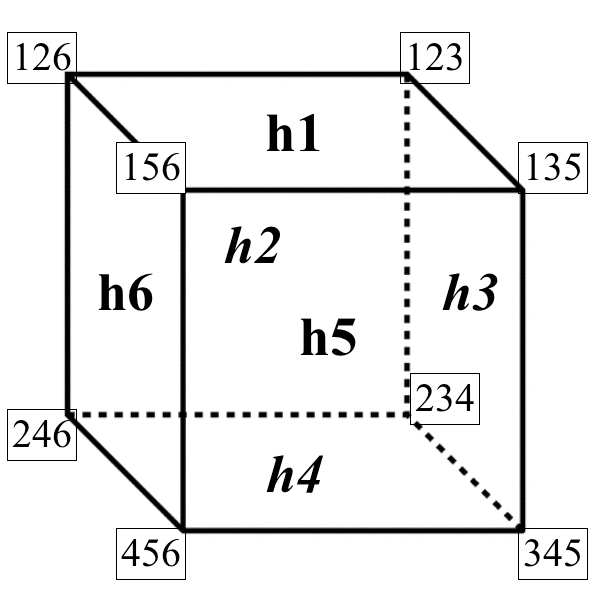}
\caption{Cube with labelled hyperplanes and separating set hyperplanes for regions corresponding to vertices.  Italicized hyperplanes are not visible from this view.
} \label{fig: cube corners}
\end{figure}
\begin{equation*}
V_{q} = \left(\begin{array}{c | c |  c  c | c  c}
1 & Q & Q_{2} & Q_{2} & Q_{3} & Q'_{3}\\
\hline
Q^{t} & v_{11} & v_{12} & v_{13} & v_{14} & v_{15} \\
\hline
Q^{t}_{2} & v_{21} & v_{22} & v_{23} & v_{24} & v_{25}\\
Q^{t}_{2} & v_{31} & v_{32} & v_{33} & v_{34} & v_{35}\\
\hline
Q^{t}_{3} & v_{41} & v_{42} & v_{43} & v_{44} & v_{45}\\
Q'^{t}_{3} & v_{51} & v_{52} & v_{53} & v_{54} & v_{55}
\end{array} \right)
\end{equation*}
where $v_{j i} = v_{ij}^t$ is a row-circulant matrix, $Q_{k} = (q^{k}, q^{k}, q^{k}, q^{k}, q^{k}, q^{k})= \mathbf{q}_{6}^{k}$, and $Q'_{k} = (q^{k}, q^{k}) = \bq_{2}^{k}$.
\begin{equation*}
\setlength{\arraycolsep}{5pt}
\setlength{\extrarowheight}{0pt}
\begin{array}{ccc}
v_{11} = C( 1, q^2 , q^2 , q^2 , q^2 , q^2)_{6 \times 6}, & v_{12} = C( q, q^{3}, q^{3}, q^{3}, q^{3}, q )_{6 \times 6},
& v_{13} = C( q, q^{3} , q^3 , q^3 , q , q^3 )_{6 \times 6}, \end{array}
\end{equation*}
\begin{equation*}
\setlength{\arraycolsep}{5pt}
\setlength{\extrarowheight}{0pt}
\begin{array}{ccc}
v_{14} = C( q^{2}, q^{4}, q^{4}, q^{4}, q^{2}, q^{2})_{6 \times 6}, & v_{15} = C(q^{2}, q^{4})_{6 \times 2}, & v_{22} = C(1, q^{2}, q^{4}, q^{4}, q^{4}, q^{2} )_{6 \times 6}, \end{array}
\end{equation*}
\begin{equation*}
\setlength{\arraycolsep}{5pt}
\setlength{\extrarowheight}{0pt}
\begin{array}{ccc}
v_{23} = C( q^{2}, q^{2}, q^{4}, q^{4}, q^{2}, q^{2} )_{6 \times 6}, & v_{24} = C( q, q^{3}, q^{5}, q^{5}, q^{3}, q )_{6 \times 6}, & v_{25} = C( q^{3}, q^{3} )_{6 \times 2},  \end{array}
\end{equation*}
\begin{equation*}
\setlength{\arraycolsep}{5pt}
\setlength{\extrarowheight}{0pt}
\begin{array}{ccc}
 v_{33} = C( 1, q^{4}, q^{2}, q^{4}, q^{2}, q^{4} )_{6 \times 6}, & v_{34} = C( q, q^{3}, q^{3}, q^{5}, q^{3}, q^{3} )_{6 \times 6}, & v_{35} = C( q, q^{5} )_{6 \times 2}, \end{array}
\end{equation*}
\begin{equation*}
\setlength{\arraycolsep}{5pt}
\setlength{\extrarowheight}{0pt}
\begin{array}{cccc}
 v_{44} = C( 1, q^{2}, q^{4}, q^{6}, q^{4}, q^{2}  )_{6 \times 6}, & v_{45} = C( q^{2}, q^{4}  )_{6 \times 2}, & v_{55} = C( 1, q^{6} )_{2\times 2}, \end{array}
\end{equation*}
where $C_{6\times 2}$ means the $6\times 2$ block of the circulant matrix.

The Smith normal form is
\setlength{\arraycolsep}{6pt}
\setlength{\extrarowheight}{-3pt}
\begin{equation*}
diag\left(1, \ (1-q^2) I_{6}, \ (1-q^2)^2 I_{12}, \ (1-q^2)^3 I_{8} \right).
\end{equation*}
The left transition matrix is $U$ and the right transition matrix is $U^{t}$, where
\setlength{\arraycolsep}{6pt}
\setlength{\extrarowheight}{0pt}

\begin{equation*}
U= \left( \begin{array}{c||c||c|c||c|c}1 & 0 & 0 & 0 & 0 & 0 \\
\hline
\hline
-q1 & I_{6} & 0 & 0 & 0 & 0 \\
\hline
\hline
q^2 1 & -q\left(I + J\right) &  I_{6} & 0 & 0 & 0 \\
\hline
q^2 1 & -q\left(I+J^2\right) & 0 &  I_{6} & 0 & 0 \\
\hline
\hline
-q^3 1 & q^2\left(I + J+J^2\right) & -q\left(I + J\right) & -qI & I_{6} & 0 \\
\hline
-q^3 1 & q^2\left(I_{2} \ | \ I_{2} \ | \ I_{2}\right) & 0 & -q\left(I_{2} \ | \ I_{2} \ | \ I_{2}\right) & 0 & I_{2} \end{array} \right);
\end{equation*}

\begin{equation*} \setlength{\extrarowheight}{-5pt}
 \left(I_{2}\ | \ I_{2} \ | \ I_{2}\right) = \left( \begin{array}{cc|cc|cc} 1 & 0 & 1 & 0 & 1 & 0\\ 0 & 1 & 0 & 1 & 0 & 1\end{array} \right), \
 \mbox{ and } \  \setlength{\extrarowheight}{-6pt} J = \left( \begin{array}{cccccc} 0 & 1 & 0 & 0 & 0 & 0\\ 0 & 0 & 1 & 0 & 0 & 0\\ 0 & 0 & 0 & 1 & 0 & 0\\ 0 &0  & 0 & 0 & 1 & 0\\ 0 & 0 & 0 & 0 & 0 & 1\\ 1 & 0 & 0 & 0 & 0 & 0\end{array} \right).
\end{equation*}
\subsection{Octahedron}
\label{chap-five}
Let $R_{0}$ denote the central region of the hyperplane arrangement
\begin{equation*}
\mathcal{A} = \{h_{1}, h_{2}, h_{3}, h_{4}, h_{5}, h_{6}, h_{7}, h_{8}\}.
\end{equation*}
\begin{equation*}
\mbox{The distance enumerator with respect to $R_{0}$ is } D_{\mathcal{A}, R_0}(t) =1 + 8t + 12t^2 + 24t^3 + 14t^4.
\end{equation*}
\begin{figure}[ht]
\includegraphics[scale=.45]{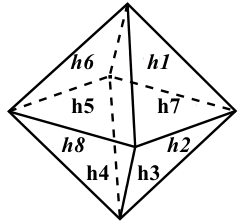}
\caption{Octahedron with labelled hyperplanes.  Hyperplanes labelled in italics are not visible from this view.}
\label{fig: oct}
\end{figure}
The hyperplanes in \fref{fig: oct} are labelled according to the order of the regions $R$ in the cube arrangement such that $\#sep(R_{0}, R) = 3$ as shown in \fref{fig: cube corners} and indexed in \eref{cube regions}.  The last six regions in the octahedron arrangement are indexed by the order of the hyperplanes in the cube arrangement in \fref{fig: cube corners}.

The $59$ regions are indexed in the following order, and denoted by the hyperplanes in $sep(R, R_0)$:
\setlength{\abovedisplayskip}{2pt}
\setlength{\belowdisplayskip}{2pt}
\setlength{\arraycolsep}{1.5pt}
\begin{gather*}
\begin{array}{c || cc cc cc| cc||}
0 & 1 & 2 & 3 & 4 & 5 & 6 & 7  & 8
\end{array}
\begin{array}{cc cc cc| cc cc cc||}
 16 & 12 & 23 & 34 & 45 & 56  & 17 & 28 & 37 & 48 & 57 & 68
 \end{array}\\
\begin{array}{ cc cc cc| cc cc cc| cc cc cc|cccccc}
 167 & 128 & 237 & 348 & 457 & 568 &
 157 & 268 & 137 & 248 & 357 & 468 &
 567 & 168 & 127 & 238 & 347 & 458 &
 156 & 126 & 123 & 234 & 345 & 456
 \end{array}\\
\begin{array}{||cccccc|cc|cc cc cc}
1267 & 1238 & 2347 & 3458 & 4567 & 1568 & 1357 & 2468 & 1567 & 1268 & 1237 & 2348 & 3457 & 4568
\end{array}
\end{gather*}
\begin{equation*}
\setlength{\extrarowheight}{-1pt}
\setlength{\arraycolsep}{4pt}
V_{q} = \left(\begin{array}{c | c  c |c  c | c c  c c | c  c c}
1 & Q_{1} & Q'_{1} & Q_{2} & Q_{2} & Q_{3} & Q_{3} & Q_{3} & Q_{3} & Q_{4} & Q'_{4} & Q_{4} \\
\hline
Q_{1}^t & v_{11} & v_{12} & v_{13} & v_{14} & v_{15} & v_{16} & v_{17} & v_{18} & v_{19} & v_{1, 10} & v_{1, 11} \\ 
Q'^t_{1} & v_{21} & v_{22} & v_{23} & v_{24} & v_{25} & v_{26} & v_{27} & v_{28} & v_{29} & v_{2, 10} & v_{2, 11} \\
\hline
Q_{2}^t & v_{31} & v_{32} & v_{33} & v_{34} & v_{35} & v_{36} & v_{37} & v_{38} & v_{39} & v_{3, 10} & v_{3, 11} \\
Q_{2}^t  & v_{41} & v_{42} & v_{43} & v_{44} & v_{45} & v_{46} & v_{47} & v_{48} & v_{49} & v_{4, 10} & v_{4, 11} \\
\hline
Q_{3}^t  & v_{51} & v_{52} & v_{53} & v_{54} & v_{55} & v_{56} & v_{57} & v_{58} & v_{59} & v_{5, 10} & v_{5, 11} \\
Q_{3}^t & v_{61} & v_{62} & v_{63} & v_{64} & v_{65} & v_{66} & v_{67} & v_{68} & v_{69} & v_{6, 10} & v_{6, 11} \\
Q_{3}^t & v_{71} & v_{72} & v_{73} & v_{74} & v_{75} & v_{76} & v_{77} & v_{78} & v_{79} & v_{7, 10} & v_{7, 11} \\
Q_{3}^t & v_{81} & v_{82} & v_{83} & v_{84} & v_{85} & v_{86} & v_{87} & v_{88} & v_{89} & v_{8, 10} & v_{8, 11} \\
\hline
Q_{4}^t & v_{91} & v_{92} & v_{93} & v_{94} & v_{95} & v_{96} & v_{97} & v_{98} & v_{99} & v_{9, 10} & v_{9, 11} \\
Q'^t_{4} & v_{10, 1} & v_{10, 2} & v_{10, 3} & v_{10, 4} & v_{10, 5} & v_{10, 6} & v_{10, 7} & v_{10, 8} & v_{10, 9} & v_{10, 10} & v_{10, 11} \\
Q_{4}^t & v_{11, 1} & v_{11, 2} & v_{11, 3} & v_{11, 4} & v_{11, 5} & v_{11, 6} & v_{11, 7} & v_{11, 8} & v_{11, 9} & v_{11, 10} & v_{11, 11} \\ 
\end{array} \right),
\end{equation*}
where $v_{j i} = v_{ij}^t$ is a row-circulant matrix, $Q_{k} =(q^k, q^k, q^k, q^k, q^k, q^k)=\bq_6^k$ and $Q'_{k}=(q^k, q^k)=\bq_2^k$.
\begin{equation*}
\setlength{\arraycolsep}{5pt}
\setlength{\extrarowheight}{0pt}
\begin{array}{ccc}
v_{11} = C( 1, q^2 , q^2 , q^2 , q^2 , q^2)_{6 \times 6}, & v_{12} = C( q^2 , q^2 )_{6 \times 2}, & v_{13} = C( q , q , q^3 , q^3 , q^3 , q^3 )_{6 \times 6}, \end{array}
\end{equation*}
\setlength{\arraycolsep}{5pt}
\setlength{\extrarowheight}{0pt}
\begin{equation*}
\begin{array}{ccc}
v_{14} = C( q , q^3 , q^3 , q^3 , q^3 , q^3 )_{6 \times 6} & v_{15} = C( q^2 , q^2 , q^4 , q^4 , q^4 , q^4 )_{6 \times 6}, & v_{16} = C( q^2 , q^4 , q^2 , q^4 , q^4 , q^4 )_{6 \times 6}, \end{array}
\end{equation*}
\setlength{\arraycolsep}{5pt}
\setlength{\extrarowheight}{0pt}
\begin{equation*}
\begin{array}{ccc}
 v_{17} = C( q^4 , q^2 , q^2 , q^4 , q^4 , q^4 )_{6 \times 6}, &  v_{18} = C( q^2 , q^2 , q^2 , q^4 , q^4 , q^4 )_{6 \times 6}, & v_{19} = C( q^3 , q^3 , q^5 , q^5 , q^5 , q^3 )_{6 \times 6},\end{array}
\end{equation*}
\setlength{\arraycolsep}{5pt}
\setlength{\extrarowheight}{0pt}
\begin{equation*}
\begin{array}{ccc}
 v_{1,10} = C( q^3 , q^5 )_{6 \times 2}, & v_{1,11} = C( q^3 , q^3 , q^3 , q^5 , q^5 , q^5 )_{6 \times 6}, & v_{22} = C( 1, q^2 )_{2 \times 2},\end{array}
\end{equation*}
\setlength{\arraycolsep}{5pt}
\setlength{\extrarowheight}{0pt}
\begin{equation*}
\begin{array}{ccc}
  v_{23} = C( q^3, q^3, q^3, q^3, q^3, q^3 )_{2 \times 6}, & v_{24} = C( q, q^3, q, q^3, q, q^3 )_{2 \times 6}, & v_{25} = C( q^2, q^4, q^2, q^4, q^2, q^4 )_{2 \times 6}, \end{array}
\end{equation*}
\setlength{\arraycolsep}{5pt}
\setlength{\extrarowheight}{0pt}
\begin{equation*}
\begin{array}{ccc}
   v_{26} = C( q^2, q^4, q^2, q^4, q^2, q^4 )_{2 \times 6}, & v_{27} = C( q^2, q^4, q^2, q^4, q^2, q^4 )_{2 \times 6}, & v_{28} = C( q^4, q^4, q^4, q^4, q^4, q^4 )_{2 \times 6}, \end{array}
\end{equation*}
\setlength{\arraycolsep}{5pt}
\setlength{\extrarowheight}{0pt}
\begin{equation*}
\begin{array}{ccc}
    v_{29} = C( q^3, q^5, q^3, q^5, q^3, q^5 )_{2 \times 6}, & v_{2,10} = C( q^3, q^5 )_{2 \times 2}, & v_{2,11} = C( q^3, q^5, q^3, q^5, q^3, q^5 )_{2 \times 6},\end{array}
\end{equation*}
\setlength{\arraycolsep}{5pt}
\setlength{\extrarowheight}{0pt}
\begin{equation*}
\begin{array}{ccc}
     v_{33} = C( 1, q^2, q^4, q^4, q^4, q^2 )_{6 \times 6}, &  v_{34} = C( q^2, q^4, q^4, q^4, q^4, q^2 )_{6 \times 6}, & v_{35} = C( q, q^3, q^5, q^5, q^5, q^3 )_{6 \times 6},\end{array}
\end{equation*}
\setlength{\arraycolsep}{5pt}
\setlength{\extrarowheight}{0pt}
\begin{equation*}
\begin{array}{ccc}
      v_{36} = C(q^3, q^3, q^3, q^5, q^5, q^3)_{6 \times 6}, & v_{37} = C(q^3, q, q^3, q^5, q^5, q^5)_{6 \times 6}, & v_{38} = C(q, q, q^3 q^5, q^5, q^3)_{6 \times 6},\end{array}
\end{equation*}
\setlength{\arraycolsep}{5pt}
\setlength{\extrarowheight}{0pt}
\begin{equation*}
\begin{array}{ccc}
       v_{39} = C(q^2, q^4, q^6, q^6, q^4, q^2)_{6 \times 6}, & v_{3, 10} =C(q^4, q^4)_{6 \times 2}, & v_{3, 11} = C(q^2, q^2, q^4, q^6, q^6, q^4)_{6 \times 6},\end{array}
\end{equation*}
\setlength{\arraycolsep}{5pt}
\setlength{\extrarowheight}{0pt}
\begin{equation*}
\begin{array}{ccc}
       v_{44}=C(1, q^4, q^2, q^4, q^2, q^4)_{6 \times 6}, & v_{45} = C(q, q^3, q^3, q^5, q^3, q^5)_{6 \times 6}, & v_{46} = C(q, q^5, q, q^5, q^3, q^5)_{6 \times 6},
      \end{array}
\end{equation*}
\begin{equation*}
\begin{array}{ccc}
       v_{47} = C(q^3, q^3, q, q^5, q^3, q^5)_{6 \times 6}, & v_{48} = C(q^3, q^3, q^3, q^5, q^5, q^5)_{6 \times 6}, & v_{49} = C( q^2, q^4, q^4, q^6, q^4, q^4)_{6 \times 6},\end{array}
\end{equation*}
\setlength{\arraycolsep}{5pt}
\setlength{\extrarowheight}{0pt}
\begin{equation*}
\begin{array}{ccc}
        v_{4, 10} = C(q^2, q^6)_{6 \times 2}, & v_{4, 11} = C(q^2, q^4, q^2, q^6, q^4, q^6)_{6 \times 6}, & v_{55} = C(1, q^4, q^4, q^6, q^4, q^4)_{6 \times 6},\end{array}
\end{equation*}
\setlength{\arraycolsep}{5pt}
\setlength{\extrarowheight}{0pt}
\noindent
\begin{equation*}
\begin{array}{ccc}
         v_{56} = C(q^2, q^4, q^2, q^6, q^4, q^4)_{6 \times 6}, & v_{57} = C(q^2, q^2, q^2, q^6, q^4, q^6)_{6 \times 6}, & v_{58} = C(q^2, q^2, q^4, q^6, q^6, q^4)_{6 \times 6}, \end{array}
\end{equation*}
\noindent
\begin{equation*}
\begin{array}{ccc}
         v_{59} = C(q, q^5, q^5, q^7, q^3, q^3)_{6 \times 6}, & v_{5, 10} = C(q^3, q^5)_{6 \times 2}, & v_{5, 11} = C(q, q^3, q^3, q^7, q^5, q^5)_{6 \times 6},\end{array}
\end{equation*}
\setlength{\arraycolsep}{5pt}
\setlength{\extrarowheight}{0pt}
\begin{equation*}
\begin{array}{ccc}
         v_{66} = C(1, q^6, q^2, q^6, q^2, q^6)_{6 \times 6}, &
        v_{67} = C( q^2, q^4, q^2, q^6, q^4, q^4)_{6 \times 6}, & v_{68} = C(q^2, q^4, q^4, q^6, q^4, q^4)_{6 \times 6}, \end{array}
\end{equation*}
\setlength{\arraycolsep}{5pt}
\setlength{\extrarowheight}{0pt}
\begin{equation*}
\begin{array}{ccc}
 v_{69} = C(q^3, q^5, q^5, q^5, q^3, q^3)_{6 \times 6}, &    v_{6, 10} = C(q, q^7)_{6 \times 2}, & v_{6, 11} = C(q, q^5, q^3, q^7, q^3, q^5)_{6 \times 6},\end{array}
\end{equation*}
\setlength{\arraycolsep}{5pt}
\setlength{\extrarowheight}{0pt}
\begin{equation*}
\begin{array}{ccc}
  v_{77} = C(1, q^4, q^4, q^6, q^4, q^4)_{6 \times 6}, &
       v_{78} = C(q^2, q^4, q^6, q^6, q^4, q^2)_{6 \times 6}, & v_{79} = C(q^3, q^7, q^5, q^5, q, q^3)_{6 \times 6},\end{array}
\end{equation*}
\setlength{\arraycolsep}{5pt}
\setlength{\extrarowheight}{0pt}
\begin{equation*}
\begin{array}{ccc}
v_{7, 10} = C(q^3, q^5)_{6 \times 2}, &
       v_{7, 11} = C(q, q^5, q^5, q^7, q^3, q^3)_{6 \times 6}, & v_{88} = C(1, q^2, q^4, q^6, q^4, q^2)_{6 \times 6},\end{array}
\end{equation*}
\setlength{\arraycolsep}{5pt}
\setlength{\extrarowheight}{0pt}
\begin{equation*}
\begin{array}{ccc}
  v_{89} = C(q^3, q^5, q^7, q^5, q^3, q)_{6 \times 6}, &   v_{8, 10} = C(q^3, q^5)_{6 \times 2}, & v_{8, 11} = C(q, q^3, q^5, q^7, q^5, q^3)_{6 \times 6}, \end{array}
\end{equation*}
\setlength{\arraycolsep}{5pt}
\setlength{\extrarowheight}{0pt}
\begin{equation*}
\begin{array}{ccc}
        v_{99} = C(1, q^4, q^4, q^8, q^4, q^4)_{6 \times 6}, & v_{9, 10} = C(q^4, q^4)_{6 \times 2}, & v_{9, 11} = C(q^2, q^2, q^2, q^6, q^6, q^6)_{6 \times 6},\end{array}
\end{equation*}
\setlength{\arraycolsep}{5pt}
\setlength{\extrarowheight}{0pt}
\begin{equation*}
\begin{array}{ccc}
         v_{10, 10} = C(1, q^2)_{2 \times 2}, & v_{10, 11} = C(q^2, q^6, q^2, q^6, q^2, q^6)_{2 \times 6}, & v_{11, 11} = C(1, q^4, q^4, q^8, q^4, q^4)_{6 \times 6} \end{array},
\end{equation*}
where $C_{6\times 2}$ means the $6\times 2$ block of the circulant matrix.

The Smith normal form is
\begin{equation*}
diag\left(1, \ (1-q^2)I_{8}, \ (1-q^2)^2 I_{24}, \ (1-q^2)^3 I_{20}, \ (1-q^2)^2(I-q^8) I_{6} \right).
\end{equation*}
The left transition matrix is $LU$ and the right transition matrix is $U^{t}R$.
\setlength{\arraycolsep}{3pt}
\setlength{\extrarowheight}{2pt}
\begin{flalign}  \label{U}
&U=\left(\begin{array}{c | c  c | c  c | c c  c  c | c  c  c}
1 & 0 & 0 & 0 & 0 & 0 & 0 & 0 & 0 & 0 & 0 & 0 \\
\hline
-q 1 & I_6 & 0 & 0 & 0 & 0 & 0 & 0 & 0 & 0 & 0 & 0 \\ 
-q 1 & 0 & I_2 & 0 & 0 & 0 & 0 & 0 & 0 & 0 & 0 & 0 \\ 
\hline
q^2 1 & -q(I + J^5) & 0 & I_6 & 0 & 0 & 0 & 0 & 0 & 0 & 0 & 0 \\ 
q^2 1 & -qI & -q\bold{I}_2^t & 0 & I_6 & 0 & 0 & 0 & 0 & 0 & 0 & 0 \\ 
\hline
0 & q^2I & 0 & -qI & -qI & I_6 & 0 & 0 & 0 & 0 & 0 & 0 \\
0 & 0 & q^2\bold{I}_2^t & 0 & -q(I+J^4) & 0 & I_6 & 0 & 0 & 0 & 0 & 0 \\
0 & -q^4I+q^2J^4 & 0 & q^3I-qJ^5 & q^3I-qJ^4 & -q^2I & 0 & I_6 & 0 & 0 & 0 & 0 \\
0 & q^2J^5 & -q^4\bold{I}_2^t & -q(I+J^5) & q^3(I+J^4) & 0 & -q^2I & 0 & I_6 & 0 & 0 & 0 \\
\hline
0 & -q^3I & 0 & q^2(I+J) & q^2I & -qI & 0 & -qJ^2 & -qJ & I_6 & 0 & 0 \\
0 & 0 & -q^3I_2 & 0 & q^2\bold{I}_2 & 0 &  -q\bold{I}_2 & 0 & 0 & 0 & I_2 & 0 \\
-q^4 1 & q^3(J^4+J^5) & 0 & -q^2J^5 & q^2I & -qI & -qI & 0 & 0 & 0 & 0 & I_6 
\end{array} \right);
\end{flalign}
\begin{equation*}
\setlength{\arraycolsep}{4pt}
\setlength{\extrarowheight}{-5pt}
\mbox{where }J = \left( \begin{array}{cccccc} 0 & 1 & 0 & 0 & 0 & 0\\ 0 & 0 & 1 & 0 & 0 & 0\\ 0 & 0 & 0 & 1 & 0 & 0\\ 0 & 0 & 0 & 0 & 1 & 0\\ 0 & 0 & 0 & 0 & 0 & 1\\ 1 & 0 & 0 & 0 & 0 & 0\end{array} \right)_{6 \times 6} \
\mbox{ and } \ \bold{I}_2=
\left( \begin{array}{cc| cc|cc}1 & 0 & 1 & 0 & 1 & 0\\
0 & 1 & 0 & 1 & 0 & 1\end{array} \right)_{2 \times 6}.
\end{equation*}
\begin{equation*}
\setlength{\extrarowheight}{-1pt}
\setlength{\arraycolsep}{6pt}
L = \left(\begin{array}{c| c  c | c  c  c}
I_{33} & 0 & 0 & 0 & 0 & 0 \\
\hline
 0 & I_6 & q^2I_6 & 0 & 0 & -qI_6 \\ 
 0 & 0 & I_6 & 0 & 0 & 0 \\ 
\hline
 0 & 0 & 0 & I_6 & 0 & 0 \\ 
  0 & 0 & 0 & 0 & I_2 & 0 \\ 
  0  & -qI_6 & q^5I_6 & 0 & 0& [2]I_6 
\end{array} \right). \
\setlength{\extrarowheight}{0pt}
\setlength{\arraycolsep}{6pt}
R=\left(\begin{array}{c| c  c | c  c  c}
I_{33} & 0 & 0 & 0 & 0 & 0 \\
\hline
 0 & I_6 & -q^4I_6 & 0 & 0 & q^5I_6 \\ 
 0 & 0 & I_6 & 0 & 0 & -qI_6 \\ 
\hline
 0 & 0 & 0 & I_6 & 0 & 0 \\ 
  0 & 0 & 0 & 0 & I_2 & 0 \\ 
  0  & 0 & -qI_6 & 0 & 0 & [2]I_6 
\end{array} \right).
\end{equation*}
where $[2]=1+q^2$.
\subsection{Pyramids}\label{chap-six}
\subsubsection{Square Base (n=4)}
\begin{figure}[ht]
\includegraphics[scale=.15]{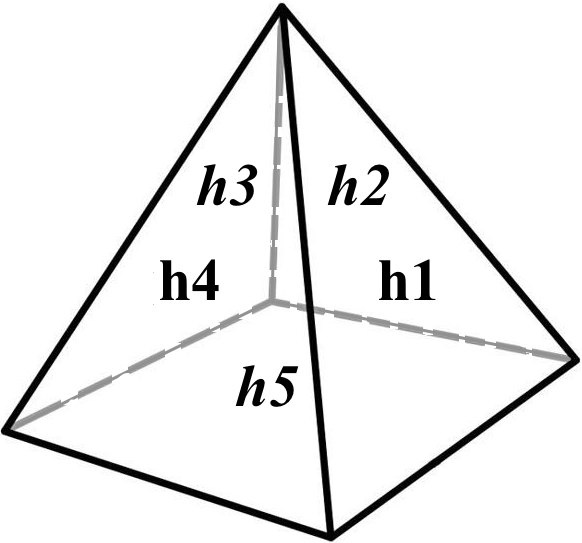}
\caption{Square base pyramid with labelled hyperplanes.  Hyperplanes labelled in italics are not visible from this view.}
\label{fig: pyr n=4}
\end{figure}
Let $\mathcal{A} = \{h_{1}, h_{2}, h_{3}, h_{4}, h_{5}\}$ be a hyperplane arrangement in $\R^3$ of a pyramid with a square base, where $h_{5}$ is the hyperplane that forms the base and the remaining hyperplanes are labelled according to \fref{fig: pyr n=4}.

Here the regions are indexed in the following order by the hyperplanes in their separating sets from $R_{0}$:
\setlength{\arraycolsep}{5pt}
\setlength{\extrarowheight}{0pt}
\begin{equation*}
\begin{array}{c|cccc|cccc|cccc|c} 5 & 15 & 25 & 35 & 45 & 125 & 235 & 345 & 145 & 1235 & 2345 & 1345 & 1245 & 12345\\
\hline
0 & 1 & 2 & 3 & 4 & 12 & 23 & 34 & 41 & 123 & 234 & 341 & 412 & 1234\end{array}
\end{equation*}
The Smith normal form of the $q$-Varchenko matrix for this arrangement is
\begin{equation*}
diag\left(1, \ (1-q^2)I_{5}, \ (1-q^2)^2 I_{10}, \ (1-q^2)^3 I_{6}, \ (1-q^2)^2(1-q^8)\right),
\end{equation*}
with the left transition matrix $WLTP$ and the right transition matrix $(TP)^{t}RW^{t}$.
\begin{equation*} \setlength{\arraycolsep}{4pt}\setlength{\extrarowheight}{-2pt}
P=\left( \begin{array}{ccc} I_{9} & -qI_{9} & 0\\
0 & I_{9} & 0\\
0 & 0 & I_{5}\end{array} \right).
\end{equation*}
\begin{equation*}
P \ V_{q}(\mathcal{A}) \ P^{t} = \left( \begin{array}{cc} (1-q^2)V_{q}(\mathcal{A}^{h_{5}}) & 0 \\ 0 & V_{q}(\mathcal{A}-\{h_{5}\})\end{array} \right) \mbox{ where } V_{q}(\mathcal{A}^{h_{5}}) = V_{q} \mbox{ for the $C_{4}$ arrangement.}
\end{equation*}
\begin{equation*}\setlength{\arraycolsep}{6pt}
\setlength{\extrarowheight}{-2.5pt}
V_{q}(\mathcal{A}-\{h_{5}\}) =
\left( \begin{array}{c|cccc|cccc|cccc|c}
1 & q & q & q & q & q^2 & q^2 & q^2 & q^2 & q^3 & q^3 & q^3 & q^3 & q^4\\
\hline
q & 1 & q^2 & q^2 & q^2 & q & q^3 & q^3 & q & q^2 & q^4 & q^2 & q^2 & q^3\\
q &  q^2 & 1 & q^2 & q^2 & q & q & q^3 & q^3 & q^2 & q^2 & q^4 & q^2 & q^3\\
q &  q^2 &  q^2 & 1 & q^2 & q^3 & q & q & q^3 & q^2 & q^2 & q^2 & q^4 & q^3\\
q &  q^2 &  q^2  & q^2 & 1 & q^3 & q^3 & q & q & q^4 & q^2 & q^2 & q^2 & q^3\\
\hline
q^2 & q & q & q^3 & q^3 & 1 & q^2 & q^4 & q^2 & q & q^3 & q^3 & q & q^2\\
q^2 & q^3 & q & q & q^3 &  q^2 &1 &  q^2 & q^4 & q & q & q^3 & q^3 & q^2\\
q^2 & q^3 & q^3 & q & q &  q^4  &  q^2 & 1 & q^2 & q^3 & q & q & q^3 & q^2\\
q^2 & q & q^3 & q^3 & q &  q^2  &  q^4 & q^2 & 1 & q^3 & q^3 & q & q & q^2\\
\hline
q^3 & q^2 & q^2 & q^2 & q^4 & q & q & q^3 & q^3 & 1 & q^2 & q^2 & q^2 & q\\
q^3 & q^4 & q^2 & q^2 & q^2 & q^3 & q & q & q^3 & q^2 & 1 & q^2 & q^2 & q\\
q^3 & q^2 & q^4 & q^2 & q^2 & q^3 & q^3 & q & q & q^2 & q^2 & 1 & q^2 & q\\
q^3 & q^2 & q^2 & q^4 & q^2 & q & q^3 & q^3 & q & q^2 & q^2 & q^2 & 1 & q\\
\hline
q^4 & q^3 & q^3 & q^3 & q^3 & q^2 & q^2 & q^2 & q^2 & q & q & q & q & 1
\end{array} \right).
\end{equation*}

Here the regions are indexed in the following order by the hyperplanes in their separating sets from $R_{0}$:
\setlength{\arraycolsep}{5pt}
\setlength{\extrarowheight}{-2pt}
\begin{equation*}
\begin{array}{c|cccc|cccc|cccc|c} 0 & 1 & 2 & 3 & 4 & 12 & 23 & 34 & 41 & 123 & 234 & 341 & 412 & 1234\end{array}
\end{equation*}
\setlength{\arraycolsep}{3pt}
\setlength{\extrarowheight}{-1pt}
\begin{flalign*}
&T = \left( \begin{array}{c|c|c|c} 1 & 0 & 0& 0\\
\hline
-q1 & I_{4} & 0& 0\\
\hline
q^21 & -q(I+J) & I_{4} & 0\\
\hline
0 & 0 & 0 & U_2 U_1 \end{array} \right), \mbox{ where }
\setlength{\arraycolsep}{4pt}
\setlength{\extrarowheight}{-1pt}
U_{1} =
 \left( \begin{array}{c|c|c|c|c}
1 & 0 & 0 & 0 & 0 \\
\hline
-q1 & I_{4} & 0 & 0 & 0\\
\hline
q^21 & -q(I+J) & I_{4}   & 0 & 0\\
\hline
0 & q^2J & -q(I+J) & I_{4} & 0\\
\hline
-q^4 & u_{41} & u_{42} & u_{43} & 1\end{array} \right),
\end{flalign*}
\begin{flalign*}
&  u_{41} = \left( \begin{array}{cccc} q^3 & 0 & 0 & q^3 \end{array} \right), \ u_{42} = \left( \begin{array}{cccc} 0 & q^2 & 0 & -q^2 \end{array} \right), \ u_{43} = \left( \begin{array}{cccc} -q & -q & 0 & 0 \end{array} \right);
\end{flalign*}
\begin{flalign*}&U_{2} = \left( \begin{array}{c|c|c|cc|c}
1 & 0 & 0 & 0 & 0 & 0\\
\hline
0 & I_{4} & 0 & 0 & 0 & 0\\
\hline
0 & 0 & I_{4}  & 0  & 0 & 0\\
\hline
0 & 0 & 0 & I_{2} & 0 & 0\\
0 & 0 & 0 & -q^2I_{2} & I_{2} & 0\\
\hline
0 & 0 & 0 & 0 & 0 & 1\end{array} \right). \
\end{flalign*}
\begin{equation*}
L = \left( \begin{array}{c|ccc} I_{20} & 0 & 0 & 0 \\
\hline
0 & 1 & q^2 & -q\\
0 & 0 & 1 & 0\\
0 & -q & q^5 & 1+q^2 \end{array} \right). \
R=  \left( \begin{array}{c|ccc} I_{20} & 0 & 0 & 0\\
\hline
0 & 1 & -q^4 & q^5\\
0 & 0 & 1 & -q\\
0 & 0 & -q & 1+q^2\end{array}\right). \
 \setlength{\extrarowheight}{-2pt}
\setlength{\arraycolsep}{6pt}
W= \left( \begin{array}{c|c|c|c|c|c|c|c}
 0 & 0 & 0 & 1 & 0 & 0 & 0 & 0\\
\hline
 1 & 0 & 0 & 0 & 0 & 0 & 0 & 0\\
\hline
 0 & 0 & 0 & 0 & I_{4} & 0 & 0 & 0\\
\hline
 0 & I_{4} & 0 & 0 & 0 & 0 & 0 & 0\\
\hline
 0 & 0 & 0 & 0 & 0 & I_{6} & 0 & 0\\
\hline
 0 & 0 & I_{4} & 0 & 0 & 0 & 0 & 0\\
\hline
 0 & 0 & 0 & 0 & 0 & 0 & I_{2} & 0\\
\hline
 0 & 0 & 0 & 0 & 0 & 0 & 0 & 1 \end{array} \right).
 \end{equation*}
\subsubsection{Pentagonal Base (n=5)}
\begin{figure}[ht]
\includegraphics[scale=.5,trim={0in 0.0in 0in 0in},clip]{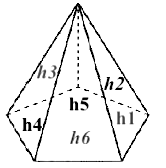}
\caption{Pentagonal base pyramid with labelled hyperplanes.  Hyperplanes labelled in italics are not visible from this view.}
\label{pyr-5}
\end{figure}
 \setlength{\extrarowheight}{-2pt}
\setlength{\arraycolsep}{4pt}
$\mathcal{A} = \{h_{1}, h_{2}, h_{3}, h_{4}, h_{5}, h_{6}\}$ is an arrangement in $\R^3$ of a pyramid with a regular pentagonal base where $h_{6}$ forms the base and the remaining hyperplanes are labelled as in \fref{pyr-5}.

%
The Smith normal form is
\begin{equation*}
diag\left( 1, \ (1-q^2)I_{6}, \ (1-q^2)^2 I_{15}, \ (1-q^2)^3I_{8}, \ (1-q^2)^2 (1-q^{10})I_{3}\right).
\end{equation*}
with the left transition matrix $WLTP$ and the right transition matrix $(TP)^{t}RW^{t}$.

Here the regions are indexed in the following order by the hyperplanes in their separating sets from $R_{0}$:
\begin{eqnarray*}
&\begin{array}{c|ccccc|ccccc} 6 & 16 & 26 & 36 & 46 & 56 & 126 & 236 & 346 & 456 & 156 \\
 \hline
0 & 1 & 2 & 3 & 4 & 5 & 12 & 23 & 34 & 45 & 15 \end{array}&\\
&\begin{array}{|ccccc| ccccc|c}
1236 & 2346 & 3456 & 1456 & 1256  & 12346 & 23456 & 13456 & 12456 & 12356 &   123456\\
\hline
 123 & 234 & 345 & 145 & 125 & 1234 & 2345 & 1345 & 1245 & 1235 & 12345
\end{array}&\\
\setlength{\extrarowheight}{-3pt}
&P=\left( \begin{array}{ccc} I_{16} & -qI_{16} & 0\\
0 & I_{16} & 0\\
0 & 0 & I_{6}\end{array} \right).&\\
%
&P \ V_{q}(\mathcal{A}) \ P^{t} = \left( \begin{array}{cc} (1-q^2)V_{q}(\mathcal{A}^{h_{6}}) & 0 \\ 0 & V_{q}(\mathcal{A}-\{h_{6}\})\end{array} \right) \mbox{ where }V_{q}(\mathcal{A}^{h_{6}}) = V_{q} \mbox{ for the $C_{5}$ arrangement.}&
\end{eqnarray*}
%
\begin{equation*}
V_{q}(\mathcal{A}-\{h_{6}\}) =
\left( \begin{array}{c c c c c c}
1 & Q & Q_{2} & Q_{3} & Q_{4} & q^{5}\\
Q^{t} & v_{11} & v_{12} & v_{13} & v_{14} & Q^{t}_{4}\\
Q^{t}_{2} & v_{21} & v_{22} & v_{23} & v_{24} & Q^{t}_{3}\\
Q^{t}_{3} & v_{31} & v_{32} & v_{33} & v_{34} & Q^{t}_{2}\\
Q^{t}_{4} & v_{41} & v_{42} & v_{43} & v_{44} & Q^{t}\\
q^{5} & Q_{4} & Q_{3} & Q_{2} & Q & 1
\end{array} \right)
\end{equation*}
where $v_{j i} = v_{ij}^t$ is a row-circulant matrix and $Q_{k} = (q^{k}, q^{k}, q^{k}, q^{k}, q^{k}) = \bq^{k}_{5}$.
\setlength{\arraycolsep}{5pt}
\setlength{\extrarowheight}{0pt}
\begin{equation*}
\begin{array}{ccc}
        v_{11} = v_{44} = C(1, q^{2}, q^{2}, q^{2}, q^{2} )_{5 \times 5}, &  v_{12} = v_{34} = C(q, q^{3}, q^{3}, q^{3}, q)_{5 \times 5}, & v_{13} = v_{24} = C(q^{2}, q^{4}, q^{4}, q^{2}, q^{2})_{5 \times 5},\end{array}
\end{equation*}
\setlength{\arraycolsep}{5pt}
\setlength{\extrarowheight}{0pt}
\begin{equation*}
\begin{array}{ccc}
        v_{14} = C( q^{3}, q^{5}, q^{3}, q^{3}, q^{3})_{5 \times 5}, & v_{22} = v_{33} = C( 1, q^{2}, q^{4}, q^{4} q^{2})_{5 \times 5}, & v_{23} = C( q, q^{3}, q^{5}, q^{3}, q)_{5 \times 5}\end{array}
\end{equation*}
Here the regions are indexed in the following order by the hyperplanes in their separating sets from $R_{0}$:
\setlength{\arraycolsep}{3pt}
\begin{equation*}
\begin{array}{c|ccccc|ccccc|ccccc| ccccc|c} 0 & 1 & 2 & 3 & 4 & 5 & 12 & 23 & 34 & 45 & 15 & 123 & 234 & 345 & 145 & 125 & 1234 & 2345 & 1345 & 1245 & 1235 & 12345 \end{array}
\end{equation*}
\begin{flalign*}
 &T = \left( \begin{array}{c|c|c|c|c} 1 & 0 & 0& 0 & 0\\
\hline
-q & I_{5} & 0& 0 &0\\
\hline
q^2 & -q(I+J) & I_{5} & 0 & 0\\
\hline
0 & q^2J^2 & -q(I+J) & I_{5}  & 0\\
\hline
0 & 0 & 0 & 0 & U_2U_1 \end{array} \right),\\
&U_{1} = \left( \begin{array}{c| c | c| c| c| c}
1 & 0 & 0 & 0 & 0 & 0\\
\hline
-q & I_{5} & 0 & 0 & 0 & 0\\
\hline
q^2 & -q(I+J) & I_{5} & 0 & 0 & 0\\
\hline
0 & q^2J & -q(I+J) & I_{5} & 0 & 0\\
\hline
0 & -q^5J^{t} & q^2J +q^4(J^{3} + J^{t}) & -q(I+J) -q^3J^{3} & I_{5} & 0\\
\hline
-q^5 & u_{51} & u_{52} & u_{53} & u_{54} & 1\end{array} \right); \\
&u_{51} = \left( \begin{array}{ccccc} q^4 & 0 & 0 & 0 & q^4 \end{array} \right),
 u_{52} = \left( \begin{array}{ccccc} 0 & 0 & 0 & 0 & -q^3 \end{array} \right), \ u_{53} = \left( \begin{array}{ccccc} 0 & q^2 & 0 & 0 & 0 \end{array} \right), \mbox{ and } \ u_{54} = \left( \begin{array}{ccccc} -q & -q & 0 & 0 & 0 \end{array} \right);
\end{flalign*}
 \begin{flalign*}
&\mbox{and } U_{2} =
\left( \begin{array}{c|c|c|c|c|c}
1 & 0 & 0 & 0 & 0 & 0\\
\hline
0 & I_{5} & 0 & 0 & 0 & 0\\
\hline
0 & 0 & I_{5} & 0 & 0 & 0\\
\hline
0 & 0 & 0 & I_{5} & 0 & 0\\
\hline
0 & 0 & 0 & 0 & I_{5} & 0\\
\hline
0 & 0 & 0 & u & 0 & 1
\end{array}\right); \ u = \left( \begin{array}{ccccc} 0 & 0 & 0 & -q^2(1-q^2) & -q^2(1-q^2) \end{array} \right).
 \end{flalign*}
 \noindent
 \setlength{\arraycolsep}{4.5pt}
\setlength{\extrarowheight}{0pt}
\begin{equation*} L = \left( \begin{array}{cc} I_{32} & 0\\ 0 & L_0 \end{array} \right), \mbox{ where }
\end{equation*}
\begin{equation*}
L_0=\left( \begin{array}{cccccc}
1 & 0 & -1 & -q^2 & 0 & 0\\
 0 & 1 & 0 & 0 & -(1+q^2) & 0\\
 0 & q^2 & 1 & -(1+q^4-q^8) & -q^2(1+q^2) & -q(1-q^4)\\
 0 & -q^3 & -q & - q(1-q^2-q^4+q^8) & q^3(1+q^2) & 1+q^2- q^6\\
 0 & -q^2(1+q^2) & -q^2 & q^2(1-q^6-q^8) & 1+ 2q^2+2q^4 + q^6 & -q^5(1+q^2)\\
 -q^2 & -2q^2 & 0 & 1+q^2-q^8 & 1+q^2+q^4 & -q(1+q^4)
\end{array}\right).
\end{equation*}
\begin{equation*}
\setlength{\extrarowheight}{0pt}
 R =  \left( \begin{array}{cc} I_{32} & 0\\ 0 & R_0 \end{array} \right),
\end{equation*}
where $R_0 =$
\begin{equation*}
\left( \begin{array}{cccccc}
 1 & q^4 & 1+q^4+q^6 + q^8 & q(2+2q^4+3q^6 + q^8 + 2q^{10} + 2q^{12} + q^{14}) & -(1-2q^2-2q^8-q^{10}) & 1-q^2+q^6\\
 0 & 1 & q^2(1+q^2) & q^3(2+q^2+q^4+2q^6+q^8) & 1+ q^2+q^4 +q^6  & -q^2(1+q^2)\\
 0 & 0 & 1 & q(1+q^6) & -(1-q^4) & 1+q^2 \\
 0 & 0 & 0 & 0 & -1 & 1+q^2\\
 0 & 0 & 0 & -q(1-q^2) & 1 - q^2 & q^2\\
 0 & 0 & 0 & 1 & q & -q(2+q^2)
 \end{array} \right).
\end{equation*}
\begin{equation*}
\setlength{\extrarowheight}{-2pt}
\setlength{\arraycolsep}{6pt}
W =\left( \begin{array}{c| c |c | c |c |c |c | c }
0 & 0 & 0 & 1 & 0 & 0 & 0 & 0\\
\hline
1 & 0 & 0 & 0 & 0 & 0 & 0 & 0\\
\hline
0 & 0 & 0 & 0 & I_{5} & 0 & 0 & 0\\
\hline
0 & I_{5} & 0 & 0 & 0 & 0 & 0 & 0\\
\hline
0 & 0 & 0 & 0 & 0 & I_{10} & 0 & 0\\
\hline
0 & 0 & I_{5} & 0 & 0 & 0 & 0 & 0\\
\hline
0 & 0 & 0 & 0 & 0 &  0 &  I_{3} & 0\\
\hline
0 & 0 & 0 & 0 & 0 &  0 &  0 & I_{3}\end{array}\right).
\end{equation*}

\section{Acknowledgements}
NJ is supported by Simons Foundation grant no. 523868. 
KCM is supported by Simons Foundation Grant no. 636482.



\end{document}